\documentclass{article}
\usepackage[english]{babel}
\usepackage{bbold}
\usepackage{verbatim}
\usepackage[utf8x]{inputenc}
\usepackage{amsmath}
\numberwithin{equation}{subsection}
\usepackage{amssymb}
\usepackage{amsthm}
\usepackage{graphicx}
\usepackage{mathtools}
\usepackage[T1]{fontenc}
\usepackage{amsthm}
\usepackage{setspace}
\usepackage{amsfonts}
\usepackage{eucal}
\usepackage{mathrsfs}

\usepackage{geometry}
\usepackage{pictexwd,dcpic}
\usepackage{tikz}
\usepackage{tqft}
\usepackage{pigpen}
\usepackage{xcolor}
\usepackage{geometry}
\usepackage[colorinlistoftodos]{todonotes} 
\usepackage[colorlinks=true,
		     allcolors=black]{hyperref}              
             
\usepackage{fancyhdr} 
\newtheorem{lmm}[subsection]{Lemma}
\newtheorem{thm}[subsection]{Theorem}
\newtheorem{prop}[subsection]{Proposition}
\newtheorem{cor}[subsection]{Corollary}
\newtheorem*{theorem-non}{Theorem}
\newtheorem*{theorem-nonf}{Th\'eor\`eme}
\newtheorem*{corollary-non}{Corollary}
\newtheorem*{corollary-nonf}{Corollaire}
\theoremstyle{definition}
\newtheorem{defn}[subsection]{Definition}
\newtheorem{context}[subsection]{Context}
\newtheorem{ex}[subsection]{Example}
\newtheorem{rmk}[subsection]{Remark}

\newtheorem{notation}[subsection]{Notation}
\newtheorem{construction}[subsection]{Construction}

\title{On the structure of dg categories of relative singularities}
\author{Massimo Pippi}
\date{}

\newcommand{\ital}[1]{\textit{#1}}

\newcommand{\oo}{\infty}
\newcommand{\aff}[2]{\mathbb{A}^{#1}_{#2}}
\newcommand{\sch}[0]{$\textup{\textbf{Sch}}_{S}$}

\newcommand{\dbarrow}[4]{
\begindc{\commdiag}[30]
\obj(0,1)[]{$#1$}
\obj(35,1)[]{$#2$}
\mor(6,0)(29,0){$#3$}[\atright,\solidarrow]
\mor(29,2)(6,2){$#4$}[\atright,\solidarrow]
\enddc
}

\newcommand{\lgm}[1]{\textup{LG}_{S}(#1)}

\newcommand{\dgcat}{\textup{dgCat}_S}
\newcommand{\dgcatlf}{\textup{dgCat}^{\textup{lf}}_S}
\newcommand{\dgcatdk}{\textup{\textbf{dgCat}}_S}
\newcommand{\dgcatm}{\textup{\textbf{dgCat}}^{\textup{idm}}_S}
\newcommand{\dgcatlfo}{\textup{dgCat}^{\textup{lf},\otimes}_S}
\newcommand{\dgcatdko}{\textup{\textbf{dgCat}}^{\otimes}_S}
\newcommand{\dgcatmo}{\textup{\textbf{dgCat}}^{\textup{idm},\otimes}_S}
\newcommand{\dgcatft}{\textup{\textbf{dgCat}}_S^{\textup{ft}}}
\newcommand{\cohb}[1]{\textup{\textbf{Coh}}^{b}(#1)}
\newcommand{\perf}[1]{\textup{\textbf{Perf}}(#1)}
\newcommand{\qcoh}[1]{\textup{\textbf{QCoh}}(#1)}
\newcommand{\cohm}[1]{\textup{\textbf{Coh}}^{-}(#1)}
\newcommand{\cohbp}[2]{\textup{\textbf{Coh}}^b(#1)_{\textup{\textbf{Perf}}(#2)}}
\newcommand{\cohmp}[2]{\textup{\textbf{Coh}}^{-}(#1)_{\textup{\textbf{Perf}}(#2)}}
\newcommand{\sing}[1]{\textup{\textbf{Sing}}(#1)}
\newcommand{\pdgcat}{\textup{Pairs-dgCat}^{\textup{lf}}_S}
\newcommand{\pdgcato}{\textup{Pairs-dgCat}^{\textup{lf},\otimes}_S}
\newcommand{\mf}[2]{\textup{\textbf{MF}}(#1,#2)}
\begin{document}
\maketitle

\begin{abstract}

In this paper we show that every object in the dg category of relative singularities $\sing{B,\underline{f}}$ associated to a pair $(B,\underline{f})$, where $B$ is a ring and $\underline{f}\in B^n$, is equivalent to an homotopy retract of a $K(B,\underline{f})$-dg module concentrated in $n+1$ degrees, where $K(B,\underline{f})$ denotes the Koszul algebra associated to $(B,\underline{f})$. When $n=1$, we show that Orlov's comparison theorem, which relates the dg category of relative singularities and that of matrix factorizations of an LG-model, holds true without any regularity assumption on the potential.
\end{abstract}
\tableofcontents

\section*{Introduction}
Let $A$ be a Noetherian (commutative) ring. A matrix factorization of a pair $(B,f)$, where $B$ is an $A$-algebra and $f\in B$, is the datum of two projective finitely generated $B$-modules $(E_0,E_1)$ together with two $B$-linear morphisms $d_0:E_0\rightarrow E_1$, $d_1:E_1\rightarrow E_0$ such that $d_1\circ d_0=f\cdot id_{E_0}$ and $d_0\circ d_1=f\cdot id_{E_1}$. These objects, introduced by D. Eisenbud in \cite{eis80}, can be organized in a $2$-periodic dg category $\mf{B}{f}$ in a natural way. On the other hand, given such a pair $(B,f)$, we can define another dg category $\sing{B,f}$, called the dg category of relative singularities of the pair. The pushforward along the inclusion $\mathfrak{i}:Spec(B)\times^h_{\aff{1}{S}}S\rightarrow Spec(B)$ induces a dg functor
\begin{equation*}
    \mathfrak{i}_*:\sing{Spec(B)\times^h_{\aff{1}{S}}S}\rightarrow \sing{Spec(B)},
\end{equation*}
where $\sing{Z}$ stands for $\cohb{Z}/\perf{Z}$, $-\times_{\aff{1}{S}}^h-$ for the derived fiber product and $S=Spec(A)$. Then $\sing{B,f}$ is defined as the kernel of this dg functor.
The connection between dg categories of relative singularities and dg categories of matrix factorizations has been first envisioned by R.O. Buchweitz and D. Orlov (see \cite{buch} and \cite{orl04}), who showed that if $B$ is a regular ring and $f$ is a regular function, then (the homotopy categories of) $\mf{B}{f}$ and $\sing{B,f}$ are equivalent. Notice that under these hypothesis $Spec(B)\times^h_{\aff{1}{S}}S=Spec(B/f)=Spec(B)\times_{\aff{1}{S}}S$ (underived fiber product) and $\sing{B,f}\simeq \sing{B/f}$\footnote{Indeed, if $X$ is an underived (Noetherian) scheme, $\sing{X}=0$ if and only if $X$ is regular.}. The dg category of relative singularities was introduced by J.~Burke and M.~Walker in \cite{bw12} and by A.~Efimov and L.~Positselski in \cite{efpo15} in order to remove the regularity hypothesis on $B$. 

In the recent paper \cite{brtv} the authors show, along the way, that these equivalences are part of a lax monoidal $\oo$-natural transformation
\begin{equation*}
    Orl^{-1,\otimes} : \sing{\bullet,\bullet} \rightarrow \mf{\bullet}{\bullet} : \lgm{1}^{op,\boxplus}\rightarrow \dgcatmo
\end{equation*}
and suggest that, in order to remove the regularity hypothesis on $f$, one should consider the derived zero locus $Spec(B)\times^h_{\aff{1}{S}}S$ instead of the classical one (see \cite[Remarks 2.50 and 2.51]{brtv}). This remark comes from the observation that if $f$ is regular the two notions coincide and that if $B$ is regular and $f=0$, one can compute that both $\mf{B}{0}$ and $\sing{B,0}\simeq \sing{Spec(B)\times^h_{\aff{1}{S}}S}$ are equivalent to $\perf{B[u,u^{-1}]}$\footnote{For the equivalence $\sing{B,0}\simeq \perf{B[u,u^{-1}]}$ see \cite[Proposition 2.45]{brtv} and/or \cite{pr11}.}, where $u$ sits in cohomological degree $2$, while the classical zero locus of $f$ coincides with $B$ and thus the associated dg category of singularities is zero.

More generally, one can consider the dg category of relative singularities of any pair $(B,\underline{f})$, where $\underline{f}\in B^n$ with $n\geq 1$, defined analogously to the case where $n=1$:
\begin{equation}
    \sing{B,\underline{f}}:=fiber\bigl ( \mathfrak{i}_*: \sing{Spec(B)\times^h_{\aff{n}{S}}S}\rightarrow \sing{Spec(B)}\bigr ).
\end{equation}
There exists an algorithm that shows that this dg category is built up from $K(B,\underline{f})$-dg modules concentrated in $n+1$ degrees, where $K(B,\underline{f})$ is the Koszul algebra associated to $(B,\underline{f})$:
\begin{theorem-non}{\textup{(\ref{structure theorem})}}
Let $(Spec(B),\underline{f})$ be a $n$-dimensional affine Landau-Ginzburg model  over $S$. Then every object in the dg category of relative singularities $\sing{B,\underline{f}}$ is an homotopy retract of an object represented by a $K(B,\underline{f})$-dg module concentrated in $n+1$ degrees.
\end{theorem-non}
Moreover, when $n=1$, the algorithm mentioned above can be used to show that
\begin{theorem-non}{\textup{(\ref{structure theorem n=1})}}
Let
\begin{equation*}
\begindc{\commdiag}[20]
\obj(-10,0)[]{$(E,d,h)= 0$}
\obj(30,0)[]{$E_m$}
\obj(60,0)[]{$E_{m+1}$}
\obj(90,0)[]{$\dots$}
\obj(120,0)[]{$E_{m'-1}$}
\obj(150,0)[]{$E_{m'}$}
\obj(180,0)[]{$0$}
\mor(3,0)(27,0){$$}
\mor(33,-1)(57,-1){$d_{m}$}[\atright, \solidarrow]
\mor(57,1)(33,1){$h_{m+1}$}[\atright, \solidarrow]
\mor(63,-1)(87,-1){$d_{m+1}$}[\atright, \solidarrow]
\mor(87,1)(63,1){$$}[\atright, \solidarrow]
\mor(93,-1)(117,-1){$$}[\atright, \solidarrow]
\mor(117,1)(93,1){$h_{m'-1}$}[\atright, \solidarrow]
\mor(123,-1)(147,-1){$d_{m'-1}$}[\atright, \solidarrow]
\mor(147,1)(123,1){$h_{m'}$}[\atright, \solidarrow]
\mor(153,0)(177,0){$$}
\enddc
\end{equation*}
be a $K(B,f)$-dg module whose associated complex of $B$-modules is strictly perfect. In other words, $(E,d,h)$ is the datum of a strictly bounded complex of projective $B$-modules $(E,d)$ together with an homotopy $h:E\rightarrow E[-1]$ between the zero endomorphism of $E$ and the one induced by multiplication by $f$, enhancing $(B,d)$ with a natural structure of $K(B,f)$-dg module. In our notation, $d_i$ (resp. $h_i$) denotes the $i^{th}$ component of the differential (resp. of the homotopy). For more details, see the discussion after Remark \ref{rmk model Koszul algebra}. Then the following equivalence holds in $\sing{B,f}$:
\begin{equation*}
(E,d,h)\simeq 
\begindc{\commdiag}[25]
\obj(0,-3)[]{$\underbracket{\bigoplus_{i\in \mathbb{Z}}E_{2i-1}}_{-1}$}
\obj(40,-3)[]{$\underbracket{\bigoplus_{i\in \mathbb{Z}}E_{2i}}_{0}$}
\obj(50,0)[]{$,$}
\mor(8,-1)(32,-1){$d+h$}[\atright, \solidarrow]
\mor(32,1)(8,1){$d+h$}[\atright, \solidarrow]
\enddc
\end{equation*}
where $d$ (resp. $h$) is the sum of the $d_i$'s (resp. $h_i$'s).
Moreover, the equivalence is natural in $(E,d,h)$.
\end{theorem-non}
It is then possible to deduce the following
\begin{corollary-non}{\textup{(\ref{Orlov equivalence})}}
The lax monoidal $\oo$-natural transformation
\begin{equation*}
Orl^{-1,\otimes} : \sing{\bullet,\bullet} \rightarrow \mf{\bullet}{\bullet} : \lgm{1}^{op,\boxplus}\rightarrow \dgcatmo
\end{equation*}
constructed in \cite[\S2.4]{brtv} defines a lax monoidal $\oo$-natural equivalence.
\end{corollary-non}
Recall that, for an affine LG model $(Spec(B),f)$, 
\begin{equation*}
    Orl^{-1}_{(B,f)} : \sing{B,f}\xrightarrow{\simeq} \mf{B}{f}
\end{equation*}
is defined as the dg functor
\begin{equation*}
    \begindc{\commdiag}[20]
    \obj(0,30)[1]{$(E,d,h)$}
    \obj(70,30)[2]{$\bigoplus_{i\in \mathbb{Z}}E_{2i-1}$}
    \obj(120,30)[3]{$\bigoplus_{i\in \mathbb{Z}}E_{2i}$} 
    \mor{1}{2}{$$}[\atright,\aplicationarrow]
    \mor(78,29)(112,29){$d+h$}[\atright,\solidarrow]
    \mor(112,31)(78,31){$d+h$}[\atright,\solidarrow]
    \obj(0,0)[4]{$(E',d',h')$}
    \obj(70,0)[5]{$\bigoplus_{i\in \mathbb{Z}}E'_{2i-1}$}
    \obj(120,0)[6]{$\bigoplus_{i\in \mathbb{Z}}E'_{2i}.$} 
    \mor{4}{5}{$$}[\atright,\aplicationarrow]
    \mor(78,-1)(112,-1){$d'+h'$}[\atright,\solidarrow]
    \mor(112,1)(78,1){$d'+h'$}[\atright,\solidarrow]
    \mor{1}{4}{$\phi$}
    \mor{2}{5}{$\oplus \phi_{2i-1}$}
    \mor{3}{6}{$\oplus \phi_{2i}$}
    \enddc
\end{equation*}

The corollary above improves all the previous results on the equivalence between the dg categories of singularities and the dg category of matrix factorizations as it removes the regularity assumption on the potential. 
\vspace{0.3cm}

\textit{Acknowledgments} This paper's results are part of my PhD project under the supervision of B. To\"en and G. Vezzosi. I would like to thank them both for uncountable many conversations on the subject. I would also like to acknowledge D. Beraldo, F.~D\'eglise, T.~Dyckerhoff, B.~Keller, V. Melani, M. Robalo, S.~Scherotzke and J.~Tapia  for many useful remarks and suggestions.

I would like to thank two anonymous referees who spotted several typos and inaccuracies and provided me with many useful suggestions to improve the exposition of this paper.

This material is based upon work supported by the National Science Foundation under Grant No. DMS-1440140 while the author was in residence at the Mathematical Sciences Research Institute in Berkeley, California, during the Spring 2019 semester.

This project has received funding from the European Research Council (ERC) under the European
Union Horizon 2020 research and innovation programme (grant agreements No. 741501 and No. 725010).

The author is supported by the collaborative research center SFB 1085 Higher Invariants - Interactions between Arithmetic Geometry and Global
Analysis funded by the Deutsche Forschungsgemeinschaft.
\section{Preliminaries}\label{preliminaries}
In this section we will introduce notation and recall some well known facts about the theory of dg categories.

For us, all rings will be commutative with an identity element. Moreover, we will always assume that rings are Noetherian, even when not explicitly mentioned.
\begin{notation}
We fix a base ring $A$. We will refer to its prime spectrum $Spec(A)$ by $S$ and to the category of $S$-schemes of finite type by \sch.
\end{notation}
We will usually identify every ordinary category with its nerve. We will therefore avoid writing $N(\mathcal{C})$ to refer to the nerve of the ordinary category $\mathcal{C}$. 
\subsection*{Reminders on dg categories}
\begin{rmk}
For more details on the theory of dg categories, we invite the reader to consult \cite{ke06}, \cite{to11} and/or \cite{ro14}.
\end{rmk}

Consider the ordinary category $\dgcat$ of small $A$-linear dg categories together with $A$-linear dg functors. Recall that a quasi-equivalence is a dg functor which induces quasi-isomorphisms on the hom complexes and such that the functor induced on the homotopy categories is essentially surjective. A crucial fact in the theory of dg categories is the existence of a cofibrantly generated model category structure on $\dgcat$ whose weak equivalences are exactly quasi-equivalences (see \cite{tab05}). The underlying $\oo$-category of this model category is the $\oo$-localization of $\dgcat$ with respect to the class of quasi-equivalences. We will denote this $\oo$-category by $\dgcatdk$. 

Another crucial class of dg functors is that of Morita equivalences: a dg functor $T\rightarrow T'$ is a Morita equivalence if it induces a quasi-equivalence on the associated derived categories of perfect dg modules. The class of quasi-equivalences is contained in that of Morita equivalences. Therefore, using the theory of Bousfield localizations we can enrich $\dgcat$ with a cofibrantly generated model category structure where weak equivalences are precisely Morita equivalences. The underlying $\oo$-category, that we will label $\dgcatm$, coincides with the $\oo$-localization of $\dgcat$ with respect to Morita equivalences. In particular we have the following couple of composable $\oo$-functors;
\begin{equation}
\dgcat \rightarrow \dgcatdk \rightarrow \dgcatm.
\end{equation}
The $\oo$-category $\dgcatm$ can be identified with the full subcategory of $\dgcatdk$ of dg categories $T$ for which the Yoneda embedding $T\hookrightarrow \hat{T}_c$ is a quasi-equivalence (see \cite[\S4.4]{to11}). These dg categories are called \emph{triangulated} (see \cite{to07}) or \emph{idempotent complete} (see \cite{ro14}). Here, $\hat{T}_c$ stands for the dg category of compact (i.e. perfect) $T^{op}$-modules. In particular, if $T\hookrightarrow \hat{T}_c$ is a quasi-equivalence, it follows that the homotopy category of $T$ is equivalent to the homotopy category of $\hat{T}_c$, which is triangulated and idempotent complete (compact objects are stable under retracts). Then the $\oo$-functor $\dgcatdk \rightarrow \dgcatm$ is a left adjoint to the inclusion $\oo$-functor and can be informally described by the assignement $T\mapsto \hat{T}_c$.

We can enhance both $\dgcatdk$ and $\dgcatm$ with a symmetric monoidal structure in such a way that, if we restrict to the full subcategory $\dgcatlf \subseteq \dgcat$ of locally flat (small) dg categories, there we get two composable symmetric monoidal $\oo$-functors
\begin{equation}
    \dgcatlfo \rightarrow \dgcatdko \rightarrow \dgcatmo.
\end{equation}
For more on Morita theory of dg categories, we refer to \cite{to07}.

Of major relevance in the following is the definition of quotient of dg categories. Given a dg category $T$ together with a full sub dg category $T'$, both of them in $\dgcatm$, we will consider the dg quotient $T/T'$ which is defined as the pushout $T\amalg_{T'}0$ in $\dgcatm$. Here $0$ stands for the final object in $\dgcatm$, i.e. the dg category with only one object and the zero hom complex. More generally, we can define the dg quotient of any morphism $T'\rightarrow T$ in $\dgcatm$ as the pushout above. A fundamental fact is that the homotopy category of $T/T'$ coincides with the Verdier quotient of $T$ by the full subcategory generated by the image of $T'$ (see \cite{dri}). The dg category $T/T'$ can also be obtained as the image in $\dgcatm$ of the pushout $T\amalg_{T'}0$ calculated in $\dgcatdk$.

We conclude this section by recalling that compact objects in $\dgcatm$ coincide with dg categories of finite type over $A$ (see \cite[Definition 2.4]{tv07} for a definition and \cite[Lemma 2.11]{tv07} for a proof of this fact). In particular, as explained for example in \cite[Proposition 6.1.27]{ro14} or in \cite[\S2.1]{tv19},
\begin{equation}
    \textup{Ind}(\dgcatft)\simeq \dgcatm.
\end{equation}
\subsection*{Higher dimensional Landau-Ginzburg models}
\begin{context}
Assume that $A$ is a local, Noetherian regular ring of finite Krull dimension.
\end{context}
Recall that the category of \ital{Landau-Ginzburg models over $S$} is the category of flat $S$-schemes of finite type together with a potential (i.e. a map to $\aff{1}{S}$). The morphisms are those morphisms of $S$-schemes which are compatible with the potential. Moreover, this category has a natural symmetric monoidal enhancement due to the fact that $\aff{1}{S}$ is an abelian group object in the category of $S$-schemes. It is very easy to generalize this category to the case where schemes are provided with multipotentials, i.e. with maps to $\aff{n}{S}$, for any $n\geq 1$.
\begin{defn}
Fix $n \geq 1$. Define the category of \ital{$n$-dimensional Landau-Ginzburg models over S} ($n$-LG models over $S$ for brevity) to be the full subcategory of $\textup{\textbf{Sch}}_{S/\mathbb{A}^{n}_S}$ spanned by those objects
\[
\begindc{\commdiag}[20]
\obj(0,20)[1]{$X$}
\obj(60,20)[2]{$\aff{n}{S}$}
\obj(30,0)[3]{$S$}
\mor{1}{2}{$\underline{f}=(f_1,\dots,f_n)$}
\mor{1}{3}{$p$}[\atright,\solidarrow]
\mor{2}{3}{$\textup{proj.}$}
\enddc
\]
where $p$ is a flat morphism. Denote this category by $\lgm{n}$ and its objects by $(X,\underline{f})$. 

For convenience, we also introduce the following (full) subcategories of $\lgm{n}$:
\begin{itemize}
\item $\lgm{n}^{\text{fl}}$, the category of \textit{flat Landau-Ginzburg models of order n over S}, spanned by those objects $(X,\underline{f})$ such that $\underline{f}$ is flat and by $(S,\underline{0})$;
\item $\lgm{n}^{\text{aff}}$, the category of \textit{affine Landau-Ginzburg models of order n over S}, spanned by those objects $(X,\underline{f})$ such that $X$ is affine;
\item $\lgm{n}^{\text{aff,fl}}$, the category of \textit{flat, affine Landau-Ginzburg models of order n over S}, spanned by those objects $(X,\underline{f})$ such that $X$ is affine and $\underline{f}$ is flat and by $(S,\underline{0})$.
\end{itemize}
\end{defn}
\begin{construction}
As in \cite{brtv}, we can enhance $\lgm{n}$ (and its variants) with a symmetric monoidal structure. Consider the "sum morphism"\footnote{notice that it is flat.}
\begin{equation}
+ : \aff{n}{S}\times_S \aff{n}{S}\rightarrow \aff{n}{S}
\end{equation}
on $\aff{n}{S}$, corresponding to 
\[
A[T_1,\dots,T_n]\rightarrow A[X_1,\dots,X_n]\otimes_A A[Y_1,\dots,Y_n]
\]
\[
T_i\mapsto X_i\otimes 1 + 1 \otimes Y_i \hspace{0.5cm} i=1,\dots,n.
\]
Then define 
\begin{equation}
\boxplus : \lgm{n}\times \lgm{n} \rightarrow \lgm{n}
\end{equation}
by the formula
\[
\bigl ( (X,\underline{f}), (Y,\underline{g})\bigr )\mapsto (X\times_S Y, \underline{f}\boxplus \underline{g}).
\]
Here, $\underline{f}\boxplus \underline{g}$ is the following composition
\[
X\times_S Y\xrightarrow{\underline{f}\times \underline{g}} \aff{n}{S}\times_S \aff{n}{S}\xrightarrow{+} \aff{n}{S}.
\]
Notice that $X\times_S Y$ is still flat over $S$, whence this functor is well defined. It is also easy to remark that $\boxplus$ is associative - i.e. there exist natural isomorphisms $\bigl ( (X,\underline{f})\boxplus (Y,\underline{g}) \bigr) \boxplus (Z,\underline{h})\simeq  (X,\underline{f})\boxplus \bigl ((Y,\underline{g}) \boxplus (Z,\underline{h})\bigr) $ - and that for any object $(X,\underline{f})$, $(S,\underline{0})\boxplus (X,\underline{f})\simeq (X,\underline{f})\simeq (X,\underline{f})\boxplus (S,\underline{0})$. More briefly, $\bigl ( \lgm{n}, \boxplus, (S,\underline{0})\bigr )$ is a symmetric monoidal category. It is not hard to see that this construction works on $\lgm{n}^{\text{fl}}$, $\lgm{n}^{\text{aff}}$ and $\lgm{n}^{\text{aff,fl}}$ too. Indeed,  this is clear for $\lgm{n}^{\text{aff}}$  and if $\underline{f}$ and $\underline{g}$ are flat morphisms, so is $\underline{f}\times \underline{g}$ and therefore $\underline{f}\boxplus \underline{g}$ is a composition of flat morphisms. 
\end{construction}
\begin{notation}
We will denote by $\lgm{n}^{\boxplus}$ (resp. $\lgm{n}^{\text{fl},\boxplus}$, $\lgm{n}^{\text{aff},\boxplus}$, $\lgm{n}^{\text{aff,fl}, \boxplus}$   ) these symmetric monoidal categories.
\end{notation}
\begin{rmk}
Notice that $\lgm{1}^{\boxplus}$ is exactly the symmetric monoidal category
$\textup{LG}_S^{\boxplus}$ defined in \cite[\S 2]{brtv}.
\end{rmk}
\begin{rmk}
Fix $n\geq 1$. Notice that the symmetric group $\mathcal{S}_n$ acts on the category of $n$-LG models over $S$. Indeed, for any $\sigma \in \mathcal{S}_n$ and for any $(X,\underline{f}) \in \lgm{n}$, we can define
\[
\sigma \cdot (X,\underline{f}) := (X,\sigma \cdot \underline{f}).
\]
\end{rmk}

\subsection*{Dg categories of singularities}
It is a classic theorem due to Auslander-Buchsbaum (\cite{ab56}) and Serre (\cite{se55}) that a Noetherian local ring $R$ is regular if and only if it has finite global dimension. This extremely important fact can be rephrased by saying that the every object in $\cohb{R}$ is equivalent to an object in $\perf{R}$. In particular, $R$ is regular if and only if $\cohb{R}/\perf{R}$ is zero. This explains why the quotient above is called \textit{category of singularities}. 

Before going on with the precise definitions, let us fix some notation.

Let $(X,\underline{f})$ be a $n$-LG model over $S$. Then consider the (derived) zero locus of $\underline{f}$, i.e. the (derived) fiber product

\begin{equation}
\begindc{\commdiag}[20]
\obj(0,15)[1]{$X_0$}
\obj(30,15)[2]{$X$}
\obj(0,-15)[3]{$S$}
\obj(30,-15)[4]{$\aff{n}{S}.$}
\mor{1}{2}{$\mathfrak{i}$}
\mor{1}{3}{$$}

\mor{2}{4}{$\underline{f}$}
\mor{3}{4}{$\text{zero}$}
\enddc
\end{equation}

\begin{rmk}
Notice that $X_0\simeq X\times_{\aff{n}{S}}^h S$ coincides with the classical zero locus of $\underline{f}$ whenever $(X,\underline{f})$ belongs to $\lgm{n}^{\text{fl}}$ (except for $(X,\underline{f})=(S,\underline{0})$). In general, we always have a closed embedding $t : X\times_{\aff{n}{S}} S=\pi_0(X_0) \rightarrow X_0$.
\end{rmk}

\begin{rmk}\label{lci morphisms stable pullbacks}
Recall that a morphism of derived schemes $f:Y\rightarrow Z$ is a locally finitely presented if the induced morphism on the truncated schemes $\pi_0(f):\pi_0(Y)\rightarrow \pi_0(Z)$ is locally finitely presented in the classical sense and if the cotangent complex $\mathbb{L}_f$ is perfect. Also recall that $f:Y\rightarrow Z$ is (derived) lci if it is locally finitely presented and if the cotangent complex $\mathbb{L}_f$ has Tor amplitude $[-1,0]$. As all these properties are clearly preserved under derived fiber products, we see that lci morphisms are closed with respect to this operation.
In particular, since $S\xrightarrow{\text{zero}} \aff{n}{S}$ is lci, we get that $\mathfrak{i}: X_0\rightarrow X$ is a lci morphism of derived schemes.

\end{rmk}
We will consider the following ($A$-linear) dg categories:
\begin{itemize}
\item $\qcoh{X}$ (resp. $\qcoh{X_0}$ ), the $A$-linear dg category of complexes of quasi-coherent sheaves on $X$ (resp. $X_0$);
\item $\perf{X}$ (resp. $\perf{X_0}$), the full sub-dg category of $\qcoh{X}$ (resp. $\qcoh{X_0}$) spanned by perfect complexes. Recall that, for a derived scheme $Z$, an object $E\in \qcoh{Z}$ is perfect if, locally, it belongs to the thick sub-dg category of $\qcoh{Z}$ spanned by $\mathcal{O}_Z$. Perfect complexes are exactly dualizable objects. In our case, they coincide with compact objects in $\qcoh{Z}$ too (see \cite{bzfn}); 
\item $\cohb{X}$ (resp $\cohb{X_0}$), the full sub-dg category of $\qcoh{X}$  (resp. $\qcoh{X_0}$ ) spanned by those cohomologically bounded complexes $E$ such that $H^*(E)$ is a coherent $H^0(\mathcal{O}_X)$ (resp. $H^0(\mathcal{O}_{X_0})$) module;
\item $\cohm{X}$ (resp $\cohm{X_0}$), the full sub-dg category of $\qcoh{X}$  (resp. $\qcoh{X_0}$ ) spanned by those cohomologically bounded above complexes with coherent cohomology groups over $H^0(\mathcal{O}_X)$ (resp. $H^0(\mathcal{O}_{X_0})$);
\item $\cohbp{X_0}{X}$, the full sub-dg category of $\cohb{X_0}$ spanned by those objects $E$ such that $\mathfrak{i}_*E$ belongs to $\perf{X}$.
\end{itemize}
\begin{rmk}
Analogously to \cite[\S2]{brtv}, we have the following inclusions
\[
\perf{X}\subseteq \cohb{X} \subseteq \cohm{X}\subseteq \qcoh{X},
\]
\[
\perf{X_0}\subseteq \cohbp{X_0}{X}\subseteq \cohb{X_0}\subseteq \cohm{X_0}\subseteq \qcoh{X_0}.
\]
Indeed, being $X$ and $X_0$ eventually coconnective (see \cite[\S4, Definition 1.1.6]{gr17}), we have the inclusions $\perf{X}\subseteq \cohb{X}$ and $\perf{X_0}\subseteq \cohb{X_0}$. Moreover, as $\mathfrak{i}$ is lci, by \cite{to12}, we have that $\mathfrak{i}_*$ preserves perfect complexes. Thus, the inclusion $\perf{X_0}\subseteq \cohbp{X_0}{X}$ holds.
\end{rmk}
\begin{rmk}
As it is explained in \cite[Remark 2.14]{brtv}, the dg categories $\perf{X}$, $\perf{X_0}$, $\cohb{X}$, $\cohb{X_0}$ and $\cohbp{X_0}{X}$ are idempotent complete. Indeed, the same argument provided in \textit{loc.cit.} for the case $n=1$ works in general.
\end{rmk}
Notice that all the results in \cite[\S 2.3.1]{brtv} are not specific of the monopotential case and they remain valid in our situation. We will recall these statements for the reader's convenience and refer to \ital{loc. cit.} for the proofs, which remain untouched.
\begin{prop}
Let $(X,\underline{f})\in \lgm{n}$. Then the inclusion functor induces an equivalence 
\begin{equation}
\cohbp{X_0}{X}\simeq \cohmp{X_0}{X}.
\end{equation}
In particular, the following square is cartesian in $\dgcatm$
\begin{equation}
\begindc{\commdiag}[20]
\obj(0,30)[1]{$\cohm{X_0}$}
\obj(60,30)[2]{$\cohm{X}$}
\obj(0,0)[3]{$\cohbp{X_0}{X}$}
\obj(60,0)[4]{$\perf{X}.$}
\mor{1}{2}{$\mathfrak{i}_*$}
\mor{3}{1}{$$}
\mor{4}{2}{$$}
\mor{3}{4}{$\mathfrak{i}_*$}
\enddc
\end{equation}
\end{prop}
We now give definitions for the relevant dg categories of singularities. The reader should be aware that there are plenty of these objects that one can consider, and we will define some of them later on. The following category, known as \ital{category of absolute singularities}, first appeared in \cite{orl04}. The following is a dg enhancement of the original definition, as it appears in \cite{brtv}.
\begin{defn}\label{absolute sing}
Let $Z$ be a derived scheme of finite type over $S$ whose structure sheaf is cohomologically bounded. The \ital{dg category of absolute singularities of $Z$} is the dg quotient (in $\dgcatm$)
\begin{equation}
\sing{Z}:= \cohb{Z}/\perf{Z}.
\end{equation}
\end{defn}
\begin{rmk}
Notice that the finiteness hypothesis on $Z$ in Definition \ref{absolute sing} is absolutely indispensable, as otherwise $\perf{Z}$ may not contained in $\cohb{Z}$.
\end{rmk}
\begin{rmk}
It is well known that, for an \ital{underived} (Noetherian) scheme $Z$, the dg category $\sing{Z}$ is zero if and only if the scheme is regular. On the other hand, when we allow $Z$ to be a derived scheme, $\sing{Z}$ may be non trivial even if the underlying scheme is regular. For example, consider $Z=Spec(A\otimes^{\mathbb{L}}_{A[T]}A)$.
\end{rmk}

Following \cite{brtv} we next consider the dg category of singularities associated to an $n$-dimesional LG-model.
\begin{defn}
Let $(X,\underline{f})\in \lgm{n}$. The \ital{dg category of singularities of $(X,\underline{f})$} is the following fiber in $\dgcatm$
\begin{equation}
\sing{X,\underline{f}}:= Ker\bigl ( \mathfrak{i}_* : \sing{X_0}\rightarrow \sing{X}\bigr ).
\end{equation}
\end{defn}
\begin{rmk}
Notice that $\sing{X,\underline{f}}$ is a full sub-dg category of $\sing{X_0}$ (see \cite[Remark 2.24]{brtv}). Moreover, these two dg categories coincide whenever $X$ is a regular $S$-scheme.
\end{rmk}
\begin{prop}{\textup{(See \cite[Proposition 2.25]{brtv})}}
Let $(X,\underline{f})$ be a $n$-dimensional LG model over $S$. Then there is a canonical equivalence
\begin{equation}
\cohbp{X_0}{X}/\perf{X_0}\simeq \sing{X,\underline{f}},
\end{equation}
where the quotient on the left is taken in $\dgcatm$. 
\end{prop}

We shall now re-propose, for the multi-potential case, the strict model for $\cohbp{X_0}{X}$ which was first introduced in \cite{brtv}.

\begin{construction}\label{strict model derived tensor product}
Let $(Spec(B),\underline{f})\in \lgm{n}^{\textup{aff}}$. Consider the Koszul complex $K(B,\underline{f})$
\begin{equation}
0\rightarrow \bigwedge^n (B\varepsilon_1\oplus \dots \oplus B\varepsilon_n) \rightarrow \dots \rightarrow \bigwedge^2 (B\varepsilon_1\oplus \dots \oplus B\varepsilon_n) \rightarrow (B\varepsilon_1\oplus \dots \oplus B\varepsilon_n) \rightarrow B \rightarrow 0
\end{equation}
concetrated in degrees $[-n,0]$. The differential is given by
\[
\bigwedge^k (B\varepsilon_1\oplus \dots \oplus B\varepsilon_n)\rightarrow \bigwedge^{k-1} (B\varepsilon_1\oplus \dots \oplus B\varepsilon_n)
\]
\[
v_1\wedge \dots \wedge v_k\mapsto \sum_{i=1}^k (-1)^{1+i}\phi(v_i) v_1\wedge \dots \wedge \hat{v_i} \wedge \dots \wedge v_k,
\]
where $\phi : B^n \rightarrow B$ is the morphism corresponding to the matrix $[f_1\dots f_n]$. Multiplication is given by concatenation. Notice that $K(B,\underline{f})$ is a cofibrant $B$-module and that we always have a truncation morphism $K(B,\underline{f})\rightarrow B/\underline{f}$, which is a quasi-isomorphism whenever $\underline{f}$ is a regular sequence.

Therefore, we can present $K(B,\underline{f})$ as the cdga $B[\varepsilon_1,\dots,\varepsilon_n]$, where the $\varepsilon_i$'s sit in degree $-1$ and are subject to the following conditions:
\[
d(\varepsilon_i)=f_i \hspace{0.5cm} i=1,\dots,n,
\]
\[
\varepsilon_i^2=0 \hspace{0.5cm} i=1,\dots,n,
\]
\[
\varepsilon_{i_1}\dots \varepsilon_{i_k}= (-1)^{\sigma} \varepsilon_{i_{\sigma(1)}}\dots \varepsilon_{i_{\sigma(k)}} \hspace{0.5cm} \{i_1,\dots,i_k\}\subseteq \{1,\dots,n\}, \sigma \in \mathcal{S}_k.
\]
\end{construction}
\begin{ex}
For instance, when $n=1$, $K(B,f)$ is the cdga $B\varepsilon \xrightarrow{f} B$ concentrated in degrees $[-1,0]$ and, when $n=2$, $K(B,(f_1,f_2))$ is the cdga $B\varepsilon_1 \varepsilon_2\xrightarrow{\begin{bmatrix} -f_2 \\ f_1 \end{bmatrix}} B\varepsilon_1 \oplus B\varepsilon_2 \xrightarrow{\begin{bmatrix} f_1 & f_2 \end{bmatrix}} B$ concentrated in degrees $[-2,0]$.
\end{ex}
\begin{rmk}\label{rmk model Koszul algebra}
Notice that $K(B,\underline{f})$ provides a model for the cdga associated to the simplicial commutative algebra $B\otimes^{\mathbb{L}}_{A[T_1,\dots,T_n]}A$. Indeed, this can be computed explicitly for $n=1$ and the general case follows from the compatibility of the Dold-Kan correspondence with (derived) tensor products.
\end{rmk}
This strict model for the derived zero locus of an affine LG model of order $n$ over $S$ gives us strict models for the relevant categories of modules too. Following \cite{brtv}:
\begin{itemize}
\item There is an equivalence of $A$-linear dg categories between $\qcoh{X_0}$ and the dg category (over $A$) of cofibrant $K(B,\underline{f})$-dg modules, which we will denote $\widehat{K(B,\underline{f})}$. A $K(B,\underline{f})$-dg module is the datum of a cochain complex of $B$-modules $(E,d)$, together with $n$ $B$-linear morphisms $h_1,\dots,h_n : E\rightarrow E[1]$ of degree $-1$ such that
\begin{equation}\label{identities h_i}
\begin{cases}
h_i^2=0 \hspace{1.1cm} i=1,\dots,n, \\
[d,h_i]=f_i \hspace{0.5cm} i=1,\dots,n, \\
[h_i,h_j]=0 \hspace{0.5cm} i,j=1,\dots,n.
\end{cases}
\end{equation}
\item $\cohb{X_0}\subseteq \qcoh{X_0}$ corresponds to the full sub-dg category of $\widehat{K(B,\underline{f})}$ spanned by those modules of cohomologically bounded amplitude and whose cohomology is coherent over $B/\underline{f}$;
\item $\perf{X_0}\subseteq \qcoh{X_0}$ corresponds to the full sub-dg category of $\widehat{K(B,\underline{f})}$ spanned by those modules which are homotopically finitely presented.
\end{itemize}
\begin{rmk}\label{commutation d and h}
Notice that, for any $K(B,\underline{f})$-dg module, for any $1\leq k \leq n$ and for any $\{ i_1,\dots,i_k\}\subseteq \{1,\dots,n\}$ (where the $i_j$'s are pairwise distinguished),  the following formula holds:
\[
[d,h_{i_1}\circ \dots \circ h_{i_k}]=\sum_{j=1}^k(-1)^{j+1} f_{i_j} h_{i_1}\circ \dots \widehat{h_{i_j}} \dots \circ h_{i_k}.
\]
\end{rmk}
\begin{rmk}
As in the mono-potential case (see \cite[Remark 2.30]{brtv}), $\mathfrak{i}_* : \qcoh{X_0}\rightarrow \qcoh{X}$   
corresponds, under these equivalences, to the forgetful functor $\widehat{K(B,\underline{f})}\rightarrow \widehat{B}$ ($K(B,\underline{f})$ is a cofibrant $B$-module). 
\end{rmk}
Recall that a complex of $B$-modules is strictly perfect if it is strictly bounded and degree-wise projective of finite type. 
We propose the following straightforward generalization of \cite[Construction 2.31]{brtv} as a strict model for $\textup{Coh}^b(X_0)_{\textup{Perf}(X)}$:
\begin{construction}
Let $\textup{Coh}^s(B,\underline{f})$ be the $A$-linear sub-dg category of $\widehat{K(B,\underline{f})}$ spanned by those modules whose image along the forgetful functor $K(B,\underline{f})-dgmod\rightarrow B-dgmod$ is a strictly perfect complex of $B$-modules. In particular, an object $E$ in $\textup{Coh}^s(B,\underline{f})$ is a degree-wise projective cochain complex of $B$-modules together with $n$ morphisms $h_1,\dots,h_n$ of degree $-1$ satisfying the identities (\ref{identities h_i}). As $A$ is a local ring, it is clear that $\textup{Coh}^s(B,\underline{f})$ is a locally flat $A$-linear dg category.
\end{construction}
\begin{lmm}\label{explicit model sing}
Let $(X,\underline{f})=(Spec(B),\underline{f})$ be a $n$-dimensional affine LG model over $S$. Then the cofibrant replacement dg functor induces an equivalence
\begin{equation}
\textup{Coh}^s(B,\underline{f})[q.iso^{-1}]\simeq \textup{Coh}^s(B,\underline{f})/\textup{Coh}^{s,acy}(B,\underline{f})\simeq \cohbp{X_0}{X},
\end{equation}
where $\textup{Coh}^{s,acy}(B,\underline{f})$ is the full sub-dg category of $\textup{Coh}^s(B,\underline{f})$ spanned by acyclic complexes. In particular, this implies that we have equivalences of dg categories
\begin{equation}
\textup{Coh}^{s}(B,\underline{f})/\textup{Perf}^s(B,\underline{f})\simeq \cohbp{X_0}{X}/\perf{X_0}\simeq \sing{X,\underline{f}},
\end{equation}
where $\textup{Perf}^s(B,\underline{f})$ is the full sub-dg category of $\textup{Coh}^s(B,\underline{f})$ spanned by those modules which are perfect over $K(B,\underline{f})$.
\end{lmm}
\begin{proof}
See \cite[Lemma 2.33]{brtv}. The same proof  holds true in our situation too.
\end{proof}

We now exhibit the functorial properties of $\textup{Coh}^s(\bullet,\bullet)$. Let $u:(Spec(C),\underline{g})\rightarrow (Spec(B,\underline{f}))$ be a morphism in $\lgm{n}^{\textup{aff}}$. Define the dg functor
\begin{equation}
u^* : \textup{Coh}^s(B,\underline{f})\rightarrow \textup{Coh}^s(C,\underline{g})
\end{equation}
by the law
\[
E\mapsto E\otimes_B C.
\]
Notice that this dg functor is well defined as $E\otimes_B C$ is strictly bounded and degree-wise $C$-projective. It is clear that if two composable morphisms 
\[
(Spec(B),\underline{f})\xrightarrow{u}(Spec(B'),\underline{f}')\xrightarrow{u'}(Spec(B''),\underline{f}'')
\]
are given, $u'^*\circ u^*\simeq (u'\circ u)^*$ are equivalent dg functors $\textup{Coh}^s(B'',\underline{f}'')\rightarrow \textup{Coh}^s(B,\underline{f})$. It is also clear that $id_{(Spec(B),\underline{f})}^*\simeq id_{\textup{Coh}^s(B,\underline{f})}$ and that this law is (weakly) associative and (weakly) unital. In other words,
\begin{equation}\label{Coh^s}
\textup{Coh}^s(\bullet,\bullet) : \lgm{n}^{\textup{aff, op}}\rightarrow \dgcatlf
\end{equation}
has the structure of a pseudo-functor.
We next produce a lax monoidal structure on this pseudo-functor. We begin by producing a map
\begin{equation}\label{right lax map Coh^s}
\textup{Coh}^s(B,\underline{f})\otimes \textup{Coh}^s(C,\underline{g})\rightarrow \textup{Coh}^s(B\otimes_A C,\underline{f}\boxplus \underline{g}),
\end{equation}
Write
\[
K(B,\underline{f})=B[\varepsilon_1,\dots,\varepsilon_n],
\]
\[
K(C,\underline{g})=C[\delta_1,\dots,\delta_n],
\]
\[
K(B\otimes_A C,\underline{f}\boxplus \underline{g})=B\otimes_A C[\gamma_1,\dots,\gamma_n],
\]
where all the $\varepsilon_i$'s, $\delta_i$'s and $\gamma_i$'s sit in degree $-1$ and are subject to the relations (\ref{strict model derived tensor product}). Consider the following morphism
\[
\phi:K(B\otimes_A C,\underline{f}\boxplus \underline{g})\rightarrow K(B,\underline{f})\otimes_A K(C,\underline{g})
\]
\[
\gamma_i \mapsto \varepsilon_i\otimes 1 + 1\otimes \delta_i, \hspace{0.5cm} i=1,\dots, n.
\]

Let $F \in \textup{Coh}^s(B,\underline{f})$, $G\in \textup{Coh}^s(C,\underline{g})$. Then the $A$-module $F\otimes_A G$ has a natural structure of $K(B,\underline{f})\otimes_A K(C,\underline{g})$-dg module. Concretely, this is the graded $B\otimes_AC$-module whose term in degree $n$ is $\bigoplus_{i+j=n}F_i\otimes_A G_j$. The differential $d_{F\otimes_AG}$ is defined, for two homogeneous elements $x\in F_i$ and $y\in G_j$, by the usual formula
\[
d_{F\otimes_AG}(x\otimes y)=d_F(x)\otimes y +(-1)^i x\otimes d_G(y).
\]
Moreover, for two such elements, the homotopy $h_{F\otimes_AG}^j:F\otimes_AG\rightarrow F\otimes_AG[1]$ is defined by the formula
\[
h_{F\otimes_AG}^j(x\otimes y)=h_F^j(x)\otimes y +(-1)^i x\otimes h_G^j(y).
\]

 We define $F\boxtimes G$ to be $F\otimes_A G$ with the $K(B\otimes_A C,\underline{f}\boxplus \underline{g})$-dg module structure induced by $\phi$.

As $F \in \textup{Coh}^s(B,\underline{f})$, $G\in \textup{Coh}^s(C,\underline{g})$ and strictly bounded complexes are stable under tensor product, we conclude that $F\boxtimes G$ is a strictly bounded complex. To see that each term is projective of finite type over $B\otimes_AC$, it suffices to observe that in degree $m$ $F\boxtimes G$ is $\oplus_{i+j=m}F_i\otimes_AG_j$, where $F_i$ is a projective $B$-module and $G_j$ is a projective $C$-module.
 
We next exhibit the lax unit\footnote{$\underline{A}$ denotes the $\otimes$-unit in $\dgcatlfo$, i.e. the dg category with one object $\bullet$ whose complex of endomorphisms $Hom_{\underline{A}}(\bullet,\bullet)$ is just $A$ in degree $0$.}
\begin{equation}\label{lax unit Coh^s}
\underline{A} \rightarrow Coh^s(A,\underline{0}).
\end{equation}
This is simply the dg functor defined by
\[
\bullet \mapsto A,
\]
where $A$ (concentrated in degree $0$) is seen as a module over $K(A,\underline{0})$ in the obvious way, i.e. the $\varepsilon_i$'s act via zero. 

This defines a (right) lax monoidal structure on (\ref{Coh^s})
\begin{equation}\label{Coh^s monoidal}
\textup{Coh}^{s,\boxtimes}(\bullet,\bullet): \lgm{n}^{\textup{aff,op}, \boxplus}\rightarrow \dgcatlfo.
\end{equation}
\begin{rmk}\label{box product preserves perfs}
Notice that if $F \in \textup{Perf}^s(B,\underline{f})$ and $G\in \textup{Perf}^s(C,\underline{g})$, then $F\boxtimes G \in \textup{Perf}^s(B\otimes_AC,\underline{f}\boxplus \underline{g})$. This follows from the observation that, if we denote by $pr_1$ (resp. $pr_2$) the projection from $V(\underline{f})\times_S V(\underline{g})\simeq Spec(K(B,\underline{f}))\times_SSpec(K(C,\underline{g}))$ to $V(\underline{f})=Spec(K(B,\underline{f}))$ (resp. $V(\underline{g})=Spec(K(C,\underline{g}))$) and we let $\phi:V(\underline{f})\times_S V(\underline{g})\rightarrow V(\underline{f}\boxplus \underline{g})=Spec(B\otimes_AC,\underline{f}\boxplus \underline{g})$ be the morphism defined above. Notice that $\phi$ is lci by Remark \ref{lci morphisms stable pullbacks}. Then 
\[
-\boxtimes - :\textup{Coh}^s(B,\underline{f})\otimes_A \textup{Coh}^s(C,\underline{g})\rightarrow \textup{Coh}^s(B\otimes_AC,\underline{f}\boxplus \underline{g})
\] 
is a model for the $\oo$-functor
\[
\phi_*\bigl ( pr_1^*(-)\otimes pr_2^*(-) \bigr ):\cohbp{V(\underline{f})}{Spec(B)}\otimes_A \cohbp{V(\underline{g})}{Spec(C)}\rightarrow \cohbp{V(\underline{f}\boxplus \underline{g})}{Spec(B\otimes_A C)}.
\]
Clearly the external tensor product $pr_1^*(-)\otimes pr_2^*(-)$ preserves perfect complexes. In order to conclude that $F\boxtimes G$ is perfect, it suffices to notice that $\phi$ is an lci morphism. In fact, this morphism naturally lives in the following diagram

\begin{equation}
\begindc{\commdiag}[15]
\obj(0,60)[1]{$V(\underline{f})\times_S^h V(\underline{g})$}
\obj(80,60)[2]{$V(\underline{f}\boxplus \underline{g})$}
\obj(160,60)[3]{$Spec(B\otimes_A C)$}
\obj(0,30)[4]{$S$}
\obj(80,30)[5]{$\aff{n}{S}$}
\obj(160,30)[6]{$\aff{n}{S}\times_S \aff{n}{S}$}
\obj(80,0)[7]{$S$}
\obj(160,0)[8]{$\aff{n}{S},$}
\mor{1}{2}{$\phi$}
\mor{2}{3}{$\psi$}
\mor{1}{4}{$$}

\mor{2}{5}{$$}

\mor{3}{6}{$\underline{f}\times \underline{g}$}
\mor{4}{5}{$\text{zero}$}
\mor{5}{6}{$\begin{bmatrix}id & -id\end{bmatrix}$}
\mor{5}{7}{$$}

\mor{6}{8}{$+$}
\mor{7}{8}{$\text{zero}$}
\enddc
\end{equation}
where all squares are (homotopy) Cartesian.
Clearly the bottom square is Cartesian. The two bigger squares, obtained by joining the two upper squares and the two rightmost squares are also obviously Cartesian. The remaining two squares are seen to be Cartesian by a double application of the following fact: suppose that we are given a commutative diagram in an $\oo$-category $\mathcal{C}$
\begin{equation*}
  \begindc{\commdiag}[15]
    \obj(-20,10)[1]{$X_1$}
    \obj(0,10)[2]{$X_2$}
    \obj(20,10)[3]{$X_3$}
    \obj(-20,-10)[4]{$X_4$}
    \obj(0,-10)[5]{$X_5$}
    \obj(20,-10)[6]{$X_6$}
    \mor{1}{2}{$$}
    \mor{2}{3}{$$}
    \mor{1}{4}{$$}
    \mor{2}{5}{$$}
    \mor{3}{6}{$$}
    \mor{4}{5}{$$}
    \mor{5}{6}{$$}
  \enddc
\end{equation*} 
and that the bigger and righmost squares are Cartesian. Then, the square on the left is Cartesian as well.
\end{rmk}
By the same technical arguments of \cite[Construction 2.34, Construction 2.37]{brtv} we produce a (right) lax monoidal $\oo$-functor
\begin{equation}\label{oo Coh^s monoidal}
\cohbp{\bullet}{\bullet}^{\otimes}: \lgm{n}^{\textup{aff,op},\boxplus}\rightarrow \dgcatmo.
\end{equation}
In order to define the lax monoidal $\oo$-functor
\begin{equation}\label{functor sing}
\sing{\bullet,\bullet}^{\otimes} : \lgm{n}^{\text{aff,op},\boxplus}\rightarrow \dgcatmo,
\end{equation}
consider the category $\pdgcat$ introduced in \emph{loc. cit.}, whose objects are pairs $(T,S)$, where T is an $A$-linear dg category and $S$ a class of morphisms in $T$. Given two objects $(T,S)$ and $(T',S')$, morphisms $(T,S)\rightarrow (T',S')$ are those dg functors $F:T\rightarrow T'$ such that $S$ is sent into $S'$. Composition and identities are defined in the obvious way.
Given a morphism $(T,S)\rightarrow (T',S')$, we say that it is a Dwyer-Kan equivalence if the underlying dg functor is so (i.e. it is a quasi-equivalence). We denote the class of Dwyer-Kan equivalences in $\pdgcat$ by $W_{DK}$.

Notice that $\pdgcat$ inherits a symmetric monoidal structure from $\dgcatlfo$ by setting $(T,S)\otimes (T',S')=(T\otimes T',S\otimes S')$. We will refer to this symmetric monoidal category by $\pdgcato$. As we are considering locally flat dg categories, it is immediate that this tensor structure is compatible with $DK$ equivalences.
For any $n$-dimensional affine LG model $(Spec(B),\underline{f})$ over $S=Spec(A)$, define $W_{\textup{Perf}^s(B,\underline{f})}$ as the class of morphisms $\bigl ( 0\rightarrow E \bigr )_{E\in \textup{Perf}^s(B,\underline{f})}$ in $\textup{Coh}^s(B,\underline{f})$. Consider the functor
\begin{equation}\label{canonical pair dg cat}
\lgm{n}^{\textup{aff,op}} \rightarrow \pdgcat
\end{equation}
\[
(Spec(B,\underline{f}))\mapsto (\textup{Coh}^s(B,\underline{f}),W_{\textup{Perf}^s(B,\underline{f})}).
\]
If $E\in \textup{Perf}^s(B,\underline{f})$ and $F\in 
\textup{Perf}^s(C,\underline{g})$, then $(0\rightarrow E)\otimes (0\rightarrow F)\in W_{\textup{Perf}^s(B,\underline{f})}\otimes W_{\textup{Perf}^s(C,\underline{g})}$ is sent to $0\rightarrow E\boxtimes F$ via (\ref{right lax map Coh^s}), which belongs to $W_{\textup{Perf}^s(B\otimes_A C,\underline{f}\boxplus \underline{g})}$ (see Remark \ref{box product preserves perfs}). Then the functor (\ref{canonical pair dg cat}) has a lax monoidal enhancement 
\begin{equation}\label{lax canonical pair dg cat}
\lgm{n}^{\textup{aff,op}, \boxplus} \rightarrow \pdgcato.
\end{equation}

By \cite[Construction 2.34 and Construction 2.37]{brtv}, there is a strongly monoidal $\oo$-functor
\begin{equation}\label{loc dg}
loc_{dg}^{\otimes} : \pdgcato [W_{DK}^{-1}] \rightarrow \dgcatdko
\end{equation}
sending a pair $(T,S)$ to the dg localization $T[S^{-1}]_{dg}$.

We finally define (\ref{functor sing}) as the following composition
\begin{equation}
\lgm{n}^{\textup{aff,op}, \boxplus} \xrightarrow{loc. \circ (\ref{lax canonical pair dg cat})}\pdgcato[W_{DK}^{-1}]\xrightarrow{(\ref{loc dg})} \dgcatdko \xrightarrow{\text{loc.}} \dgcatmo.
\end{equation}
Notice that $(Spec(B),\underline{f})\in \lgm{n}^{\textup{aff}}$ is sent to $\sing{B,\underline{f}}$ by Lemma \ref{explicit model sing} and by the fact that the quotient $\textup{Coh}^s(B,\underline{f})/\textup{Perf}^s(B,\underline{f})$ is, by definition, the dg localization $\textup{Coh}^s(B,\underline{f})[W_{\textup{Perf}^s(B,\underline{f})}^{-1}]$ (see \cite[\S8.2]{to07}).
\begin{rmk}
If $n=1$, the lax monoidal structure on the $\oo$-functor $\sing{\bullet,\bullet}^{\otimes}$ identifies with the lax monoidal structure on the $\oo$-functor defined in \cite[Proposition 2.45]{brtv}. In fact, for every affine LG model $(B,f)$, there is a canonical morphism $(\textup{Coh}^s(B,f),W_{\textup{q.iso}})\rightarrow (\textup{Coh}^s(B,f),W_{\textup{Perf}^s(B,f)})$ in the category $\textup{Pairs-dgCat}_S^{\textup{lf}}$, where $W_{q.iso}$ denotes the class of quasi-isomorphisms in $\textup{Coh}^s(B,f)$. This naturally induces a lax monoidal natural transformation
\begin{equation*}
   \cohbp{\bullet}{\bullet}^{\otimes} \rightarrow \sing{\bullet,\bullet}^{\otimes}.
\end{equation*}
Similarly, the canonical dg functor $\perf{A[u]}\rightarrow \perf{A[u,u^{-1}]}$, where $u$ sits in cohomological degree 2, induces a lax monoidal natural transformation
\begin{equation*}
\cohbp{\bullet}{\bullet}^{\otimes} \rightarrow \cohbp{\bullet}{\bullet}^{\otimes}\otimes_{A[u]}A[u,u^{-1}],
\end{equation*}
where the lax monoidal $\oo$-functor on the right is the one defined in \cite{brtv}. The claim follows from the observation that these two lax monoidal natural transformations share the same universal property: they are universal among those lax monoidal natural transformations $\cohbp{\bullet}{\bullet}^{\otimes}\rightarrow F^{\otimes}$ such that the composition with $\perf{\bullet}^{\otimes}\rightarrow \cohbp{\bullet}{\bullet}^{\otimes}$ is homotopic to zero. Here $\perf{\bullet}^{\otimes}$ denotes the symmetric monoidal $\oo$-functor 
\begin{equation*}
\lgm{1}^{\textup{aff,op},\boxplus}\xrightarrow{-\times^h_{\aff{1}{S}}S} \textbf{dAffSch}_S^{\textup{op},\times_S} \xrightarrow{\perf{\bullet}^{\otimes}} \dgcatmo,
\end{equation*}
where $\textbf{dAffSch}_S$ is the $\oo$-category of affine derived schemes over $S$.
\end{rmk}

\section{The structure of \texorpdfstring{$\sing{B,\underline{f}}$}{sing(B,f)}}

In this section we will prove that, in the category of relative singularities $\sing{B,\underline{f}}$ associated to a $n$-dimensional affine Landau-Ginzburg model over $S$, every object is an homotopy retract of an object that can be represented by a $K(B,\underline{f})$-dg module concentrated in $n+1$-degrees. We begin with the following observation:
\begin{lmm}\label{cone K(B,f) dg mod}
Let $\phi : (E,d,h)\rightarrow (E',d',h')$ be a cocycle-morphism of $K(B,f)$-dg modules\footnote{Here $d$ (resp. $d'$) stands for the differential and $h^i$ (resp.$h'^i$ ) stands for the action of $\varepsilon_i$, where 

$K(B,\underline{f})=0\rightarrow \underbracket{B \varepsilon_1\dots \varepsilon_n}_{-n}\rightarrow \dots \rightarrow \underbracket{B\varepsilon_1\oplus \dots \oplus B\varepsilon_n}_{-1}\rightarrow \underbracket{B}_{0}$.}.  Then the cone of $\phi$ is given by
\begin{equation}
\begindc{\commdiag}[20]
\obj(-20,2)[]{$\dots \leftrightarrows$}
\obj(0,0)[]{$\underbracket{E_{n+1}\oplus E'_{n}}_{n}$}
\obj(80,0)[]{$\underbracket{E_{n+2}\oplus E'_{n+1}}_{n+1}$}
\obj(160,0)[]{$\underbracket{E_{n+3}\oplus E'_{n+2}}_{n+2}$}
\obj(185,2)[]{$\leftrightarrows \dots$}
\mor(10,2)(70,2){$\begin{bmatrix}
-d_{n+1} & 0 \\
\phi_{n+1} & d'_{n}
\end{bmatrix}$}[\atright, \solidarrow]
\mor(90,2)(150,2){$\begin{bmatrix}
-d_{n+2} & 0 \\
\phi_{n+2} & d'_{n+1}
\end{bmatrix}$}[\atright, \solidarrow]
\mor(150,4)(90,4){$\begin{bmatrix}
-h_{n+3}^{i} & 0 \\
0 & h_{n+2}^{'i}
\end{bmatrix}$}[\atright, \solidarrow]
\mor(70,4)(10,4){$\begin{bmatrix}
-h_{n+2}^{i} & 0 \\
0 & h_{n+1}^{'i}
\end{bmatrix}$}[\atright, \solidarrow]
\enddc
\end{equation}
\end{lmm}
\begin{proof}
Note that the underlying complex of $B$-modules is the classical cone. It only remains to check that all the morphisms involved in the proof of the fact that this complex of $B$ modules is the cone are compatible with the action of $\varepsilon$. This is a tedious but elementary verification.
\end{proof}

Consider an object $(E,d,\{h^i\}_{i\in \{1,\dots,n\}}) \in \textup{Coh}^s(B,\underline{f})$. Then its underlying $B$-dg module $(E,d)$ is strictly perfect. As the (derived) pullback preserves perfect complexes, $(E,d)\otimes_B K(B,\underline{f})$ lies is $\textup{Perf}^s(B,\underline{f})$. This is the $K(B,\underline{f})$-dg module which, in degree $m$ and $m+1$ has the shape
\begin{equation}
\bigoplus_{k=0}^n E_{m+k}\otimes_B\bigwedge^k \bigl (B\varepsilon_1\oplus \dots \oplus B\varepsilon_n \bigr ) \xrightarrow{\partial_m} \bigoplus_{k=0}^n E_{m+k+1}\otimes_B\bigwedge^k \bigl ( B\varepsilon_1\oplus \dots \oplus B\varepsilon_n \bigr ). 
\end{equation}
Moreover, $\partial_m$ is defined as follows: for any $x\in E_{m+k}$
\begin{equation}
\partial_m(x\otimes \varepsilon_{i_1}\wedge \dots \wedge \varepsilon_{i_k})= (-1)^k d_{m+k}(x)\otimes \varepsilon_{i_1}\wedge \dots \wedge \varepsilon_{i_k}+\sum_{j=1}^k \bigl ((-1)^{j+1}f_{i_j}x \otimes \varepsilon_{i_1}\wedge \dots \wedge \widehat{\varepsilon_{i_j}} \wedge \dots \wedge \varepsilon_{i_k} \bigr ).
\end{equation}
The degree $-1$ morphisms 
\begin{equation}
\eta^j_{m} :\bigoplus_{k=0}^n E_{m+k}\otimes_B\bigwedge^k \bigl (B\varepsilon_1\oplus \dots \oplus B\varepsilon_n \bigr ) \rightarrow \bigoplus_{k=0}^n E_{m+k-1}\otimes_B\bigwedge^k \bigl (B\varepsilon_1\oplus \dots \oplus B\varepsilon_n \bigr ) 
\end{equation}
are defined, for $x\in E_{m+k}$, by
\begin{equation}
\eta_m^j(x\otimes \varepsilon_{i_1}\wedge \dots \wedge \varepsilon_{i_k})=x\otimes \varepsilon_j\wedge \varepsilon_{i_1}\wedge \dots \wedge \varepsilon_{i_k}.
\end{equation}
These $B$-linear morphisms are the components of the $B$-linear morphisms $\eta^j:(E,d)\otimes_B K(B,\underline{f})\rightarrow (E,d)\otimes_B K(B,\underline{f})[-1]$ which endow the complex of $B$-modules $(E,d)\otimes_BK(B,\underline{f})$ with the structure of a $K(B,\underline{f})$-dg module.
Notice that we have a morphism of $B$-dg modules $\phi : (E,d)\otimes_B K(B,\underline{f})\rightarrow (E,d,\{h^i\}_{i\in \{1,\dots,n\}})$ which is defined in degree $m$ by ($x\in E_{m+k}$)
\begin{equation}
\bigoplus_{k=0}^n E_{m+k}\otimes_B \bigwedge^k \bigl (B\varepsilon_1\oplus \dots \oplus B\varepsilon_n \bigr ) \rightarrow E_m
\end{equation}
\[
x\otimes \varepsilon_{i_1}\wedge \dots \wedge \varepsilon_{i_k} \mapsto h_{m-1}^{i_1}\circ \dots \circ h_{m+k}^{i_k}(x),
\]
where with this notation, when $k=0$, we just mean the identity morphism.
\begin{lmm}
$\phi$ is a cocycle morphism of $K(B,\underline{f})$-dg modules.
\end{lmm}
\begin{proof}
It is clear that $\phi$ is a morphism of $K(B,\underline{f})$-dg modules, i.e. that
\[
\phi \circ \eta^j=h^j\circ \phi.
\]
We then only need to show that $\phi$ commutes with the differentials too. Pick $x\in E_{m+k}$. Then
\[
d_{m}(\phi_m(x\otimes \varepsilon_{i_1}\wedge \dots \wedge \varepsilon_{i_k}))=d_{m}(h^{i_1}_{m+1}\circ \dots \circ h^{i_k}_{m+k}(x))\underbracket{=}_{(\ref{commutation d and h})} 
\]
\[
\sum_{j=1}^k(-1)^{j+1}f_{i_j}h^{i_1}_{m+2}\circ \dots \circ \widehat{h^{i_j}} \circ \dots \circ h^{i_k}_{m+k}(x)+(-1)^kh^{i_1}_{m+2}\circ \dots \circ h^{i_k}_{m+k+1}\circ d_{m+k}(x).
\]
On the other hand, we have that
\[
\phi_{m+1}(\partial_m(x\otimes \varepsilon_{i_1}\wedge \dots \varepsilon_{i_k}))=
\]
\[
\phi_{m+1}\Bigl ((-1)^k d_{m+k}(x)\otimes \varepsilon_{i_1}\wedge \dots \wedge \varepsilon_{i_k}+\sum_{j=1}^k \bigl ((-1)^{j+1}f_{i_j}x \otimes \varepsilon_{i_1}\wedge \dots \wedge \widehat{\varepsilon_{i_j}} \wedge \dots \wedge \varepsilon_{i_k} \bigr ) \Bigr )
\]
\[
=(-1)^kh^{i_1}_{m+2}\circ \dots \circ h^{i_k}_{m+k+1}\circ d_{m+k}(x)+\sum_{j=1}^k(-1)^{j+1}f_{i_j}h^{i_1}_{m+2}\circ \dots \circ \widehat{h^{i_j}} \circ \dots \circ h^{i_k}_{m+k}(x).
\]
If $k=0$, then $\phi_m(x)=x$ and there is nothing to show.
\end{proof}
\begin{rmk}
As the source of $\phi$ is a perfect $K(B,\underline{f})$-dg module, it follows that $(E,d,\{h^i\}_{i=1,\dots,n})$ and $cone(\phi)$ are equivalent in $\sing{B,\underline{f}}$.
\end{rmk}
\begin{prop}
Assume that $(E,d,\{h^i\}_{i=1,\dots,n})$ as above is concentrated in degrees $[m',m]$, where $m-m'\geq n+1$ (i.e. the dg-module is concentrated in at least $n+2$-degrees). Then $cone(\phi)$ is equivalent, in $\sing{B,\underline{f}}$, to a $K(B,\underline{f})$-dg module concentrated in degrees $[m',m-1]$.
\end{prop}
\begin{proof}
We claim that we can exhibit $cone(\phi)$ as the cone of a cocycle morphism of $K(B,\underline{f})$-dg modules whose domain is 
\begin{equation}
-\bigl ( ( E_{m'}\xrightarrow{d_{m'}}\dots \xrightarrow{d_{m'+n}}E_{m'+n})\otimes_B K(B,\underline{f})\bigr ),
\end{equation}
which is a perfect $K(B,\underline{f})$-dg module concentrated in degrees $[m'-n,m'+n]$. For us, the $-$ in front of a $K(B,\underline{f})$-dg module $(U,\delta,\{\mu^j\}_{j=1,\dots,n})$ means that we change the sign of all the $\delta_i$'s and $\mu_i^j$'s. Notice that $-\bigl ( ( E_{m'}\xrightarrow{d_{m'}}\dots \xrightarrow{d_{m'+n}}E_{m'+n})\otimes_B K(B,\underline{f})\bigr )$ is a $K(B,\underline{f})$-sub-dg module of $cone(\phi)$ and that its differential and its homotopies coincide with the ones induced by this inclusion.

Now consider the $K(B,\underline{f})$-sub-dg module of $cone(\phi)$, which in degree $s$ is the projective $B$-module
\begin{equation}
E_{s} \oplus \Bigl (\bigoplus_{\substack{j=0 \\\ j+s\geq m'+n}}^n E_{j+s+1}\otimes_B \bigwedge^{j}(B\varepsilon_1\oplus \dots \oplus B\varepsilon_n) \Bigr)\subseteq \bigl ( cone(\phi) \bigr )_s,
\end{equation}
which we will refer to as $(F,\partial,\{\eta^i\}_{i=1,\dots n})$\footnote{Clearly, $\partial_.$ and $\{\eta_.^i\}_{i=1,\dots,n}$ are induced by $cone(\phi)$.}. This is still a $K(B,\underline{f})$-dg module as $\partial_.$ and $\{\eta_.^i\}_{i=1,\dots,n}$ are well defined on it, i.e. $\partial_s(F_s)\subseteq F_{s+1}$ and $\eta_{s}^i(F_s)\subseteq F_{s-1}$.  Notice that, for $s\leq m'-1$, $F_s=0$ and, for $s\geq m'+n$, $F_s=cone(\phi)_s$. This means that $(F,\partial,\{\eta^i\}_{i=1,\dots n})$ is a $K(B,\underline{f})$-dg module concentrated in degrees $[m',m]$, as $cone(\phi)_s=0$ if $s>m$. Label $\iota : (F,\partial,\{\eta^i\}_{i=1,\dots n})\rightarrow cone(\phi)$ the canonical inclusion and $\pi : cone(\phi)\rightarrow (F,\partial,\{\eta^i\}_{i=1,\dots n})$ the canonical projection.

Notice that, for any $s$, we have that
\begin{equation}
F_s\oplus  \bigl (- ( E_{m'}\xrightarrow{d_{m'}}\dots \xrightarrow{d_{m'+n}}E_{m'+n}\bigr )\otimes_B K(B,\underline{f})\bigr )_{s+1}=
\end{equation}
\[
E_s \oplus \Bigl (\bigoplus_{\substack{j=0 \\\ j+s\geq m'+n}}^n E_{j+s+1}\otimes_B \bigwedge^{j}(B\varepsilon_1\oplus \dots \oplus B\varepsilon_n) \Bigr) \oplus \Bigl ( \bigoplus_{\substack{j=0 \\\ j+s+1\leq m'+n}}^{n} E_{j+s+1}\otimes_B \bigwedge^j ( B\varepsilon_1\oplus \dots \oplus B\varepsilon_n) \Bigr )
\]
\[
= E_s \oplus \Bigl (\bigoplus_{j=0}^n E_{j+s+1}\otimes_B\bigwedge^{j}(B\varepsilon_1\oplus \dots \oplus B\varepsilon_n) \Bigr)=cone(\phi)_s.
\]
Define 
\begin{equation}
\psi : -\bigl ( ( E_{m'}\xrightarrow{d_{m'}}\dots \xrightarrow{d_{m'+n}}E_{m'+n-1})\otimes_B K(B,\underline{f})\bigr ) \rightarrow (F,\partial, \{\eta^i\}_{i=1,\dots,n})
\end{equation}
in every degree as the composition
\[
\Bigl ( \bigoplus_{\substack{j=0 \\\ j+s\leq m'+n}} E_{s+j}\otimes_B\bigwedge^j ( B\varepsilon_1\oplus \dots \oplus B\varepsilon_n) \Bigr ) \subseteq cone(\phi)_{s-1}\xrightarrow{\partial_{s-1}}cone(\phi)_s\xrightarrow{\pi} F_s.
\]
This is a cocycle morphism by construction. Notice that, as $\eta_s^i(F_s)\subseteq F_{s-1}$ and $\eta_s^i\bigl ( ( E_{m'}\xrightarrow{d_{m'}}\dots \xrightarrow{d_{m'+n}}E_{m'+n})\otimes_B K(B,\underline{f})\bigr )_s\subseteq \bigl ( ( E_{m'}\xrightarrow{d_{m'}}\dots \xrightarrow{d_{m'+n}}E_{m'+n})\otimes_B K(B,\underline{f})\bigr )_{s-1}$, by Lemma \ref{cone K(B,f) dg mod} we find that $cone(\psi)=cone(\phi)$.

To conclude, notice that $(F,\partial,\{\eta^i\}_{i=1,\dots,n})$ coincides, in degrees $m-1$ and $m$, with
\[
E_{m-1}\oplus E_m \xrightarrow{\begin{bmatrix} d_{m-1}, 1\end{bmatrix}} E_m.
\]
Therefore, $(F,\partial,\{\eta^i\}_{i=1,\dots,n})$ is quasi-isomorphic to 
\[
\begindc{\commdiag}[20]
\obj(0,0)[]{$F_{m'}$}
\obj(50,0)[]{$\dots$}
\obj(100,0)[]{$F_{m-2}$}
\obj(150,0)[]{$Ker(\begin{bmatrix} d_{m-1}, 1\end{bmatrix})\simeq E_{m-1}.$}
\mor(2,-2)(28,-2){$\partial_{m'}$}[\atright,\solidarrow]
\mor(28,2)(2,2){$\eta_{m'+1}^i$}[\atright,\solidarrow]
\mor(72,-2)(98,-2){$\partial_{m-3}$}[\atright,\solidarrow]
\mor(98,2)(72,2){$\eta_{m-2}^i$}[\atright,\solidarrow]
\mor(102,-2)(128,-2){$\tilde{\partial}_{m-2}$}[\atright,\solidarrow]
\mor(128,2)(102,2){$\tilde{\eta}_{m-1}^i$}[\atright,\solidarrow]
\enddc
\]
As $(F,\partial,\{\eta^i\}_{i=1,\dots,n})$ is equivalent to $(E,d\{h^i\}_{i=1,\dots,n})$ in $\sing{B,\underline{f}}$, we have proved the proposition. 
\end{proof}
\begin{rmk}{[B.~Keller]}
The previous proposition holds in a more general setting. More precisely: let $B$ be a commutative ring, $n$ an integer and let $R$ be a (possibly non-commutative) dg algebra such that
\begin{itemize}
    \item $R_p=0$ if $p\notin [-n,0]$;
    \item $R_0=B$.
\end{itemize}
Let $m,m'\in \mathbb{Z}$ such that $m-m'\geq n+1$ and let $E$ be a right dg $R$-module such that
\begin{itemize}
    \item $E_p=0$ if $p\notin [m',m]$;
    \item $E_p$ is a projective right $B$-module of finite type, for all $p\in \mathbb{Z}$.
\end{itemize}
Let $\textup{\textbf{Mod}}_R$ denote the dg category of right dg $R$-modules and $\perf{R}$ the sub-dg category spanned by perfect right dg $R$-modules. Then $E$ is equivalent in $\textup{\textbf{Mod}}_R/\perf{R}$ to a right dg $R$-module $F$ verifying the following two conditions:
\begin{itemize}
    \item $F_p=0$ if $p\notin [m',m-1]$;
    \item $F_p$ is a projective right $B$-module of finite type, for all $p\in \mathbb{Z}$.
\end{itemize}
\end{rmk}
\begin{rmk}{[B.~Keller]} It seems that the content of the previous proposition is close to the "fundamental domain" theorem due to C.~Amiot (see \cite{am09}) and generalized by O.~Iyama and D.~Yang (see \cite{iy17}). The precise comparison will be investigated elsewhere.
\end{rmk}
Then the following structure theorem holds:
\begin{thm}\label{structure theorem}
Let $(Spec(B),\underline{f})$ be a $n$-dimensional affine Landau-Ginzburg model  over $S$. Then every object in the dg category of relative singularities $\sing{B,\underline{f}}$ is an homotopy retract of an object represented by a $K(B,\underline{f})$-dg module concentrated in $n+1$ degrees.
\end{thm}
\begin{proof}\footnote{The author thanks an anonymous referee for suggestions on how to improve the exposition and clarity of an earlier version of this proof, that led to the present one.}
The homotopy category of $\sing{B,\underline{f}}$ coincides with the idempotent completion of the Verdier quotient of the homotopy category of $\cohbp{Spec(K(B,\underline{f}))}{Spec(B)}$ by the homotopy category of $\perf{Spec(K(B,\underline{f}))}$. If fact, by Lemma \ref{explicit model sing}, there is an equivalence
\begin{equation*}
  Q:=\frac{\cohbp{Spec(K(B,\underline{f}))}{Spec(B)}}{\perf{K(B,\underline{f})}}\simeq \sing{B,\underline{f}}.
\end{equation*}
Therefore, by results due to B.~Keller and V.~Drinfeld, we get an equivalence (see \cite{ke99,dri} and also \cite[Proposition 3.7]{ro14})
\begin{equation*}
  [Q]^{\omega}\simeq [\sing{B,\underline{f}}] ,
\end{equation*}
where, $[T]$ denotes the homotopy category of the dg category $T$ and $(-)^{\omega}$ stands for the idempotent completion. Every object in $Q$ is represented by an element of $\textup{Coh}^s(B,\underline{f})$. Assume that $(E,d,\{h^i\}_{i=1,\dots,n})$ is an object in $\textup{Coh}^s(B,\underline{f})$ concentrated in degrees $[m',m]$. We produce an inductive argument on the amplitude $a=m-m'+1$ of the interval $[m',m]$ where $(E,d,\{h^i\}_{i=1,\dots,n})$ is nonzero. If $m-m'\leq n$ there is nothing to prove. Otherwise, apply the previous proposition to replace $(E,d,\{h^i\}_{i=1,\dots,n})$ with an equivalent (in $\sing{B,\underline{f}}$) $K(B,\underline{f})$-dg module concentrated in degrees $[m',m-1]$. Hence, every object in $\sing{B,\underline{f}}$ is an homotopy retract of an object that can be represented by a $K(B,\underline{f})$-dg module concentrated in $n+1$ degrees.
\end{proof}

\section{Orlov's theorem}
\subsection*{Matrix factorizations}
It is well known (see \cite{orl04}, \cite{bw12}, \cite{efpo15}, \cite{brtv}) that the dg category of relative singularities $\sing{B,f}$ associated to a $1$-dimensional affine flat Landau-Ginzburg model over a regular local ring is equivalent to the dg category $\mf{B}{f}$, whose objects are matrix factorizations introduced by Eisenbud (\cite{eis80}).  In this section we shall recall what matrix factorizations are.

\begin{context}
In this section we will always work in the context of $1$-dimensional LG models. Therefore, we will omit to say it explicitly.
\end{context}
Let $(Spec(B),f)$ be an affine LG model over $S$.
\begin{defn}
A \textit{matrix factorization} over $(B,f)$ is the datum of a pair of projective $B$-modules of finite type $E_0$, $E_1$ together with $B$-linear morphisms $E_0\xrightarrow{p_0} E_{1}$ and $E_1\xrightarrow{p_1} E_0$ such that $p_1\circ p_0=f$ and $p_0\circ p_1=f$.
\end{defn}
We can naturally organize matrix factorizations into a $\mathbb{Z}/2\mathbb{Z}$-graded dg category $\mf{B}{f}$ as follows:
\begin{itemize}
\item the objects of $\mf{B}{f}$ are matrix factorizations over $(B,f)$;
\item given two matrix factorizations $(E,p)$ and $(F,q)$ over $(B,f)$, we define the morphisms in degree $0$ (resp. $1$) $Hom^{0}\bigl ( (E,p),(F,q) \bigr )$ (resp. $Hom^{1}\bigl ( (E,p),(F,q) \bigr )$) as the $B$-module of pairs of $B$-linear morphisms $(\phi_0 : E_0 \rightarrow F_0, \phi_1 : E_1 \rightarrow F_1)$ (resp. $(\psi_0 : E_0 \rightarrow F_1, \psi_1 : E_0 \rightarrow F_1)$);
\item given a map $(\chi_0,\chi_1): (E,p)\rightarrow (F,q)$ of degree $i$ ($i=0,1$), we define $\delta ((\chi_0,\chi_1)):= q \circ \chi -(-1)^i \chi \circ p$;
\item composition and identities are defined in the obvious way.
\end{itemize}

Then we can view $\mf{B}{f}$ as an $A$-linear dg category by means of the structure morphism $A\rightarrow B$.
\begin{rmk}
Notice that since we are considering projective $B$-modules and $B$ is flat over $A$, $\mf{B}{f}$ is a locally flat $A$-linear dg category.
\end{rmk}
The homotopy category of $\mf{B}{f}$ has a triangulated structure: the suspension is defined as
\begin{equation}
\bigl ( \dbarrow{E_0}{E_1}{p_0}{p_1}\bigr )[1]= \dbarrow{E_1}{E_0}{-p_1}{-p_0}
\end{equation}
and the cone of a closed morphism $(\phi):(E,p)\rightarrow (F,q)$ is defined by
\begin{equation}
\dbarrow{F_0\oplus E_1}{F_1\oplus E_0}{\begin{bmatrix} q_0 & \phi_1\\ 0 & -p_1 \end{bmatrix}}{\begin{bmatrix} q_1 & \phi_0\\ 0 & -p_0 \end{bmatrix}}.
\end{equation}
See \cite{orl04} for more details.
Moreover, $\mf{B}{f}$ has a symmetric monoidal structure, defined by
\begin{equation}
(E,p)\otimes (F,q)=\begindc{\commdiag}
\obj(0,1)[]{$(E_0\otimes_B F_0)\oplus (E_1\otimes_B F_1)$}
\obj(50,1)[]{$(E_0\otimes_B F_1)\oplus (E_1\otimes_B F_0).$}
\mor(18,-1)(32,-1){$p\otimes q$}[\atright,\solidarrow]
\mor(32,3)(18,3){$p\otimes q$}[\atright,\solidarrow]
\enddc
\end{equation}
As explained in \cite[Constructions 2.8, 2.10, 2.11]{brtv}, it is possible to define a lax monoidal $\oo$-functor
\begin{equation}
\mf{\bullet}{\bullet}^{\otimes}: \lgm{1}^{\text{aff,op},\boxplus}\rightarrow \textbf{dgcat}_A^{\text{idem},\otimes}.
\end{equation}
It is then possible to extend it to $\lgm{1}^{op,\boxplus}$ by Kan extension. With a little abuse of notation, we still denote this extension by
\begin{equation}\label{mf functor}
    \mf{\bullet}{\bullet}^{\otimes}: \lgm{1}^{\text{op},\boxplus}\rightarrow \textbf{dgcat}_A^{\text{idem},\otimes}.
\end{equation}
We refer to \cite[page 661]{brtv} for more details.
\begin{rmk}
There exists a second definition of matrix factorizations for non-affine LG-models $(X,f)$, see \cite{bw12}, \cite{efi18}, \cite{orl12}. If $X$ is a separated scheme with enough vector bundles, the two definitions agree.
\end{rmk}
\begin{rmk}
Being a lax monoidal $\oo$-functor, $(\ref{mf functor})$ factors through $\textup{Mod}_{\mf{A}{0}}(\dgcatm)^{\otimes}$.
\end{rmk}

\subsection*{More on the structure of $\sing{B,f}$}
As the Koszul algebra $K(B,f)$ is particularly simple, in the case $n=1$ it is possible to give a more detailed description of the objects of $\sing{B,f}$. This is what we will do in the following.

Our first remark concerns the periodicity of the dg category $\sing{B,f}$.
\begin{lmm}\label{lemma salto}
Let $\bigl (\dbarrow{\underbracket{E}_{n-1}}{\underbracket{F}_{n}}{p}{q}\bigr )$ be an object in $\textup{Coh}^s(B,f)$. Then it is equivalent to \\  $\bigl (\dbarrow{\underbracket{F}_{n-2}}{\underbracket{E}_{n-1}}{q}{p}\bigr )$ in $\sing{B,f}$.
\end{lmm}
\begin{proof}
Consider $\Bigl ( E\xrightarrow{p}F \Bigr )\otimes_B K(B,f) \in \textup{Perf}^s(B,f)$. This is the $K(B,f)$ dg-module
\begin{equation}
\begindc{\commdiag}[10]
\obj(0,-3)[]{$\underbracket{E}_{n-2}$}
\obj(80,-3)[]{$\underbracket{E\oplus F}_{n-1}$}
\obj(160,-3)[]{$\underbracket{F}_{n}.$}
\mor(3,0)(72,0){$\begin{bmatrix}
-p\\
f
\end{bmatrix}$}[\atright,\solidarrow]
\mor(89,0)(157,0){$\begin{bmatrix}
f & p
\end{bmatrix}$}[\atright,\solidarrow]
\mor(72,4)(3,4){$\begin{bmatrix}
0 & 1
\end{bmatrix}$}[\atright,\solidarrow]
\mor(157,4)(89,4){$\begin{bmatrix}
1 \\
0
\end{bmatrix}$}[\atright,\solidarrow]
\enddc
\end{equation}
Then let $\phi$ be the following morphism of $K(B,f)$ dg modules:
\begin{equation}
\begindc{\commdiag}[30]
\obj(0,20)[1]{$E$}
\obj(30,20)[2]{$F\oplus E$}
\obj(60,20)[3]{$F$}
\obj(30,0)[4]{$E$}
\obj(60,-1)[5]{$\underbracket{F}_{n}.$}
\mor(3,19)(27,19){$\begin{bmatrix}
-p \\
f
\end{bmatrix}$}[\atright, \solidarrow]
\mor(33,19)(57,19){$\begin{bmatrix}
f & p
\end{bmatrix}$}[\atright, \solidarrow]
\mor(57,21)(33,21){$\begin{bmatrix}
1 \\
0
\end{bmatrix}$}[\atright, \solidarrow]
\mor(27,21)(3,21){$\begin{bmatrix}
0 & 1
\end{bmatrix}$}[\atright, \solidarrow]
\mor(33,-1)(57,-1){$p$}[\atright, \solidarrow]
\mor(57,1)(33,1){$q$}[\atright, \solidarrow]
\mor{2}{4}{$\begin{bmatrix}
q & 1
\end{bmatrix}$}[\atright, \solidarrow]
\mor{3}{5}{$1$}[\atright, \solidarrow]
\enddc
\end{equation}
This morphism exhibits an equivalence in $\sing{B,f}$ between $\dbarrow{E}{F}{p}{q}$ and $cone(\phi)$, which is
\begin{equation}
\begindc{\commdiag}[30]
\obj(0,0)[1]{$E$}
\obj(30,0)[2]{$F\oplus E$}
\obj(60,0)[3]{$F\oplus E$}
\obj(90,-1)[4]{$\underbracket{F}_{n}$}
\mor(3,-1)(27,-1){$\begin{bmatrix}
p\\
-f
\end{bmatrix}$}[\atright, \solidarrow]
\mor(33,-1)(57,-1){$\begin{bmatrix}
-f & -p\\
q & 1
\end{bmatrix}$}[\atright, \solidarrow]
\mor(63,-1)(87,-1){$\begin{bmatrix}
1 & p
\end{bmatrix}$}[\atright, \solidarrow]
\mor(27,1)(3,1){$\begin{bmatrix}
0 & -1
\end{bmatrix}$}[\atright, \solidarrow]
\mor(57,1)(33,1){$\begin{bmatrix}
-1 & 0\\
0 & 0
\end{bmatrix}$}[\atright, \solidarrow]
\mor(87,1)(63,1){$\begin{bmatrix}
0 \\
q
\end{bmatrix}$}[\atright, \solidarrow]
\enddc
\end{equation}
and can be written as the cone of the following morphism of $K(B,f)$ dg-modules
\begin{equation}
\begindc{\commdiag}[30]
\obj(0,20)[1]{$E$}
\obj(30,20)[2]{$E$}
\obj(0,0)[3]{$F$}
\obj(30,0)[4]{$F\oplus E$}
\obj(60,-1)[5]{$\underbracket{F}_{n}.$}
\mor(3,19)(27,19){$f$}[\atright, \solidarrow]
\mor(27,21)(3,21){$1$}[\atright, \solidarrow]
\mor{1}{3}{$p$}[\atright, \solidarrow]
\mor{2}{4}{$\begin{bmatrix}
-p \\
1
\end{bmatrix}$}[\atright, \solidarrow]
\mor(3,-1)(27,-1){$\begin{bmatrix}
-f \\
q
\end{bmatrix}$}[\atright, \solidarrow]
\mor(33,-1)(57,-1){$\begin{bmatrix}
1 & p
\end{bmatrix}$}[\atright, \solidarrow]
\mor(27,1)(3,1){$\begin{bmatrix}
-1 & 0
\end{bmatrix}$}[\atright, \solidarrow]
\mor(57,1)(33,1){$\begin{bmatrix}
0 \\
q
\end{bmatrix}$}[\atright, \solidarrow]
\enddc
\end{equation}
Notice that the source of this morphism is $E\otimes_B K(B,f)$, where $E$ is a complex concentrated in degree $n-1$. In particular, as $E$ is a projective $B$-module, it is a perfect $K(B,f)$ dg-module. Therefore, in $\sing{B,f}$, the target of this morphism is equivalent to $\dbarrow{E}{F}{p}{q}$. Then consider the following morphism of $K(B,f)$ dg-modules:
\begin{equation}
\begindc{\commdiag}[30]
\obj(0,20)[1]{$F$}
\obj(30,20)[2]{$F\oplus E$}
\obj(60,20)[3]{$F$}
\obj(0,0)[4]{$F$}
\obj(30,-2)[5]{$\underbracket{E}_{n-1}.$}

\mor(3,19)(27,19){$\begin{bmatrix}
-f \\
q
\end{bmatrix}$}[\atright, \solidarrow]
\mor(33,19)(57,19){$\begin{bmatrix}
1 & p
\end{bmatrix}$}[\atright, \solidarrow]
\mor(27,21)(3,21){$\begin{bmatrix}
-1 & 0
\end{bmatrix}$}[\atright, \solidarrow]
\mor(57,21)(33,21){$\begin{bmatrix}
0 \\
q
\end{bmatrix}$}[\atright, \solidarrow]
\mor(3,-1)(27,-1){$q$}[\atright, \solidarrow]
\mor(27,1)(3,1){$p$}[\atright, \solidarrow]
\mor{4}{1}{$1$}[\atright, \solidarrow]
\mor(30,0)(30,20){$\begin{bmatrix}
-p\\
1
\end{bmatrix}$}[\atright, \solidarrow]
\enddc
\end{equation}
It is not hard to verify that this is a quasi-isomorphism. Following the chain of equivalences in $\sing{B,f}$ we get that 
\[
\dbarrow{\underbracket{E}_{n-1}}{\underbracket{F}_{n}}{p}{q}\simeq \dbarrow{\underbracket{F}_{n-2}}{\underbracket{E}_{n-1}}{q}{p}.
\]
\end{proof}

\begin{cor}\label{suspension in Sing}
Let $\bigl ( \dbarrow{\underbracket{E}_{n-1}}{\underbracket{F}_{n}}{p}{q}\bigr )$ be in $\textup{Coh}^s(B,f)$. Then 
\begin{equation}
\bigl ( \dbarrow{\underbracket{E}_{n-1}}{\underbracket{F}_{n}}{p}{q} \bigr )[1]\simeq \bigl ( \dbarrow{\underbracket{F}_{n-1}}{\underbracket{E}_{n}}{-q}{-p} \bigr )
\end{equation}
in $\sing{B,f}$.
\end{cor}
\begin{proof}
In $\textup{Coh}^s(B,f)$, we know that $\bigl ( \dbarrow{\underbracket{E}_{n-1}}{\underbracket{F}_{n}}{p}{q} \bigr )[1]$ is equivalent to \\$\bigl ( \dbarrow{\underbracket{E}_{n-2}}{\underbracket{F}_{n-1}}{-p}{-q} \bigr )$. Then, by Lemma \ref{lemma salto} we get 
\[
\bigl ( \dbarrow{\underbracket{E}_{n-1}}{\underbracket{F}_{n}}{p}{q} \bigr )[1]\simeq \bigl ( \dbarrow{\underbracket{E}_{n-2}}{\underbracket{F}_{n-1}}{-p}{-q} \bigr ) \simeq \bigl ( \dbarrow{\underbracket{
F}_{n-1}}{\underbracket{E}_{n}}{-q}{-p} \bigr ).
\]
\end{proof}

We will now provide an explicit description of the image of an object via the quotient functor 
\[
\textup{Coh}^s(B,f)\rightarrow \sing{B,f}
\]
\begin{thm}\label{structure theorem n=1}
Let
\begin{equation}
\begindc{\commdiag}[20]
\obj(-10,0)[]{$(E,d,h)= 0$}
\obj(30,0)[]{$E_m$}
\obj(60,0)[]{$E_{m+1}$}
\obj(90,0)[]{$\dots$}
\obj(120,0)[]{$E_{m'-1}$}
\obj(150,0)[]{$E_{m'}$}
\obj(180,0)[]{$0$}
\mor(3,0)(27,0){$$}
\mor(33,-1)(57,-1){$d_{m}$}[\atright, \solidarrow]
\mor(57,1)(33,1){$h_{m+1}$}[\atright, \solidarrow]
\mor(63,-1)(87,-1){$d_{m+1}$}[\atright, \solidarrow]
\mor(87,1)(63,1){$$}[\atright, \solidarrow]
\mor(93,-1)(117,-1){$$}[\atright, \solidarrow]
\mor(117,1)(93,1){$h_{-1}$}[\atright, \solidarrow]
\mor(123,-1)(147,-1){$d_{-1}$}[\atright, \solidarrow]
\mor(147,1)(123,1){$h_0$}[\atright, \solidarrow]
\mor(153,0)(177,0){$$}
\enddc
\end{equation}
be an object in $\textup{Coh}^s(B,f)$. Then the following equivalence holds in $\sing{B,f}$:
\begin{equation}
(E,d,h)\simeq 
\begindc{\commdiag}[25]
\obj(0,-3)[]{$\underbracket{\bigoplus_{i\in \mathbb{Z}}E_{2i-1}}_{-1}$}
\obj(40,-3)[]{$\underbracket{\bigoplus_{i\in \mathbb{Z}}E_{2i}}_{0},$}
\mor(8,-1)(32,-1){$d+h$}[\atright, \solidarrow]
\mor(32,1)(8,1){$d+h$}[\atright, \solidarrow]
\enddc
\end{equation}
where $d$ (resp. $h$) is the sum of the $d_i$'s (resp. $h_i$'s).
Moreover, it is natural in $(E,d,h)$,
\end{thm}
\begin{proof}
We will assume that $m=-2n+1$ for some $n>0$ (if $m=-2n+2$, just put $E_{-2n+1}=0$) and that $m'=0$. It is clear that this does not compromise the generality of the proof.

We shall prove that 
\begin{equation*}
   (E,d,h)\simeq 
\begindc{\commdiag}[25]
\obj(0,-3)[]{$\underbracket{\bigoplus_{i\in \mathbb{Z}}E_{2i-1}}_{-1}$}
\obj(40,-3)[]{$\underbracket{\bigoplus_{i\in \mathbb{Z}}E_{2i}}_{0}$}
\mor(8,-1)(32,-1){$d+h$}[\atright, \solidarrow]
\mor(32,1)(8,1){$d+h$}[\atright, \solidarrow]
\enddc
\end{equation*}
by an induction argument on $-m$. If $-m\leq 1$, there is nothing to prove. Let $-m>1$ and assume that for every $K(B,f)$-dg module $(F,\delta, \chi)$ concentrated in at most $-m-1$ degrees, there is a natural equivalence
\begin{equation*}
(F,\delta, \chi)\simeq 
\begindc{\commdiag}[25]
\obj(0,-3)[]{$\underbracket{\bigoplus_{i\in \mathbb{Z}}F_{2i-1}}_{-1}$}
\obj(40,-3)[]{$\underbracket{\bigoplus_{i\in \mathbb{Z}}F_{2i}}_{0}$}
\mor(8,-1)(32,-1){$\delta+\chi$}[\atright, \solidarrow]
\mor(32,1)(8,1){$\delta+\chi$}[\atright, \solidarrow]
\enddc
\end{equation*}
in $\sing{B,f}$.

The first part of the proof is the same as the one of Theorem \ref{structure theorem}, but we rewrite it in an explicit manner for the reader's convenience. 

Consider the perfect $K(B,f)$-dg module $(E,d)\otimes_B K(B,f)$\footnote{recall that $(E,d)$ is a perfect $B$-dg module.} and the following morphism \\$\phi:(E,d)\otimes_B K(B,f)\rightarrow (E,d,h)$ of $K(B,f)$-dg modules:
\begin{equation}
\begindc{\commdiag}[20]
\obj(0,30)[1]{$\underbracket{E_{-2n+1}}_{-2n}$}
\obj(40,30)[2]{$\underbracket{E_{-2n+2}\oplus E_{-2n+1}}_{-2n+1}$}
\obj(100,30)[3]{$\underbracket{E_{-2n+3}\oplus E_{-2n+2}}_{-2n+2}$}
\obj(140,30)[4]{$\dots$}
\obj(170,30)[5]{$\underbracket{E_{0}\oplus E_{-1}}_{-1}$}
\obj(200,30)[6]{$\underbracket{E_{0}}_{0}$}
\mor(5,29)(25,29){$\small{\begin{bmatrix}
-d_{-2n+1} \\
f
\end{bmatrix}}$}[\atright, \solidarrow]
\mor(25,31)(5,31){$\small{\begin{bmatrix}
0 & 1
\end{bmatrix}}$}[\atright, \solidarrow]
\mor(60,29)(80,29){$\small{\begin{bmatrix}
-d_{-2n-2} & 0 \\
f & d_{-2n+1}
\end{bmatrix}}$}[\atright, \solidarrow]
\mor(80,31)(60,31){$\small{\begin{bmatrix}
0 & 1 \\
0 & 0
\end{bmatrix}}$}[\atright, \solidarrow]
\mor(117,29)(140,29){$\small{\begin{bmatrix}
-d_{-2n+3} & 0 \\
f & d_{-2n+2}
\end{bmatrix}}$}[\atright, \solidarrow]
\mor(140,31)(117,31){$\small{\begin{bmatrix}
0 & 1 \\
0 & 0
\end{bmatrix}}$}[\atright, \solidarrow]
\mor(140,29)(160,29){$$}
\mor(160,31)(140,31){$$}
\mor(180,29)(197,29){$\small{\begin{bmatrix}
f & d_{-1}
\end{bmatrix}}$}[\atright, \solidarrow]
\mor(197,31)(180,31){$\small{\begin{bmatrix}
1 \\
0
\end{bmatrix}}$}[\atright, \solidarrow]
\obj(40,-10)[7]{$\underbracket{E_{-2n+1}}_{-2n+1}$}
\obj(100,-10)[8]{$\underbracket{E_{-2n+2}}_{-2n+2}$}
\obj(140,-10)[9]{$\dots$}
\obj(170,-10)[10]{$\underbracket{ E_{-1}}_{-1}$}
\obj(200,-10)[11]{$\underbracket{E_{0}}_{0}.$}
\mor{2}{7}{$\small{\begin{bmatrix}
h_{-2n+2} & 1
\end{bmatrix}}$}
\mor{3}{8}{$\small{\begin{bmatrix}
h_{-2n+3} & 1
\end{bmatrix}}$}
\mor{5}{10}{$\small{\begin{bmatrix}
h_{0} & 1
\end{bmatrix}}$}
\mor{6}{11}{$1$}
\mor(45,-11)(95,-11){$d_{-2n+1}$}[\atright, \solidarrow]
\mor(95,-9)(45,-9){$h_{-2n+2}$}[\atright, \solidarrow]
\mor(105,-11)(125,-11){$$}
\mor(125,-9)(105,-9){$$}
\mor(145,-11)(165,-11){$$}
\mor(165,-9)(145,-9){$$}
\mor(175,-11)(195,-11){$d_{-1}$}[\atright, \solidarrow]
\mor(195,-9)(175,-9){$h_{0}$}[\atright, \solidarrow]
\enddc
\end{equation}
Then, $Cone(\phi)$ is equivalent to $(E,d,h)$ in $\sing{B,f}$. $Cone(\phi)$ is the $K(B,f)$-dg module
\begin{equation}
\begindc{\commdiag}[18]
\obj(-20,0)[]{$\small{\underbracket{E_{-2n+1}}_{-2n-1}}$}
\obj(30,0)[]{$\small{\underbracket{E_{-2n+2}\oplus E_{-2n+1}}_{-2n}}$}
\obj(120,0)[]{$\small{\underbracket{E_{-2n+3}\oplus E_{-2n+2}\oplus E_{-2n+1}}_{-2n+1}}$}
\mor(-12,1)(12,1){$\small{\begin{bmatrix}
d_{\footnotesize{-2n+1}} \\
-f
\end{bmatrix}}$}[\atright, \solidarrow]
\mor(12,5)(-12,5){$\small{\begin{bmatrix}
0 & -1
\end{bmatrix}}$}[\atright, \solidarrow]
\mor(48,1)(92,1){$\small{\begin{bmatrix}
d_{\footnotesize{-2n+2}} & 0 \\
-f & -d_{-2n+1} \\
h_{-2n+2} & 1
\end{bmatrix}}$}[\atright, \solidarrow]
\mor(92,5)(48,5){$\small{\begin{bmatrix}
0 & -1 & 0 \\
0 & 0 & 0
\end{bmatrix}}$}[\atright, \solidarrow]
\mor(150,1)(180,1){$\small{\begin{bmatrix}
d_{\footnotesize{-2n+3}} & 0 & 0 \\
-f & -d_{-2n+2} & 0 \\
h_{-2n+3} & 1 & d_{-2n+1}
\end{bmatrix}}$}[\atright, \solidarrow]
\mor(180,5)(150,5){$\small{\begin{bmatrix}
0 & -1 & 0 \\
0 & 0 & 0 \\
0 & 0 & h_{-2n+2}
\end{bmatrix}}$}[\atright, \solidarrow]
\obj(0,-50)[]{$\small{\underbracket{E_{-2n+4}\oplus E_{-2n+3}\oplus E_{-2n+2}}_{-2n+2}}$}
\obj(50,-50)[]{$\dots$}
\obj(90,-50)[]{$\small{\underbracket{E_{0}\oplus E_{-1}\oplus E_{-2}}_{-2}}$}
\obj(150,-50)[]{$\small{\underbracket{E_0\oplus E_{-1}}_{-1}}$}
\obj(180,-50)[]{$\small{\underbracket{E_0}_{0}},$}
\mor(30,-49)(48,-49){$$}[\atright, \solidarrow]
\mor(48,-45)(30,-45){$$}[\atright, \solidarrow]
\mor(52,-49)(70,-49){$$}[\atright, \solidarrow]
\mor(70,-45)(52,-45){$$}[\atright, \solidarrow]
\mor(108,-49)(140,-49){$\small{\begin{bmatrix}
-f & -d_{-1} & 0 \\
h_0 & 1 & d_{-2}
\end{bmatrix}}$}[\atright, \solidarrow]
\mor(140,-45)(108,-45){$\small{\begin{bmatrix}
-1 & 0 \\
0 & 0 \\
0 & h_{-1}
\end{bmatrix}}$}[\atright, \solidarrow]
\mor(160,-49)(178,-49){$\small{\begin{bmatrix}
1 & d_{-1}
\end{bmatrix}}$}[\atright, \solidarrow]
\mor(178,-45)(160,-45){$\small{\begin{bmatrix}
0 \\
h_0
\end{bmatrix}}$}[\atright, \solidarrow]
\enddc
\end{equation}
which can be seen as the cone of the following morphism (call it $\varphi$)
\begin{equation}
\begindc{\commdiag}[18]
\obj(0,0)[1]{$\small{\underbracket{E_{-2n+1}}_{-2n}}$}
\obj(50,0)[2]{$\small{\underbracket{E_{-2n+2}\oplus E_{-2n+1}}_{-2n+1}}$}
\obj(130,0)[3]{$\small{\underbracket{E_{-2n+2}}_{-2n+2}}$}
\mor(8,1)(32,1){$\small{\begin{bmatrix}
-d_{-2n+1} \\
f
\end{bmatrix}}$}[\atright, \solidarrow]
\mor(32,5)(8,5){$\small{\begin{bmatrix}
0 & 1
\end{bmatrix}}$}[\atright, \solidarrow]
\mor(68,1)(122,1){$\small{\begin{bmatrix}
f & d_{-2n+1}
\end{bmatrix}}$}[\atright, \solidarrow]
\mor(122,5)(68,5){$\small{\begin{bmatrix}
1\\
0
\end{bmatrix}}$}[\atright, \solidarrow]
\obj(50,-50)[4]{$\small{\underbracket{E_{-2n+3}\oplus E_{-2n+1}}_{-2n+1}}$}
\obj(130,-50)[5]{$\small{\underbracket{E_{-2n+4}\oplus E_{-2n+3}\oplus E_{-2n+2}}_{-2n+2}}$}
\obj(180,-50)[]{$\small{\dots}$}
\obj(200,-50)[]{$\small{\underbracket{E_0}_{0}}.$}
\mor(68,-49)(102,-49){$\small{\begin{bmatrix}
d_{-2n+3} & 0 \\
-f & 0 \\
h_{-2n+3} & d_{-2n+1}
\end{bmatrix}}$}[\atright, \solidarrow]
\mor(102,-45)(68,-45){$\small{\begin{bmatrix}
0 & -1 & 0 \\
0 & 0 & h_{-2n+2}
\end{bmatrix}}$}[\atright, \solidarrow]
\mor(158,-49)(178,-49){$$}[\atright, \solidarrow]
\mor(178,-45)(158,-45){$$}[\atright, \solidarrow]
\mor(182,-49)(198,-49){$$}[\atright, \solidarrow]
\mor(198,-45)(182,-45){$$}[\atright, \solidarrow]
\mor{2}{4}{$\small{\begin{bmatrix}
d_{-2n+2} & 0 \\
h_{-2n+1} & 1
\end{bmatrix}}$}[\atright, \solidarrow]
\mor{3}{5}{$\small{\begin{bmatrix}
0 \\
-d_{-2n+2} \\
1
\end{bmatrix}}$}
\enddc
\end{equation}
As the source of this morphism is $\bigl ( \underbracket{E_{-2n+1}}_{-2n+1}\xrightarrow{d_{-2n+1}} \underbracket{E_{-2n+2}}_{-2n+2} \bigr )\otimes_B K(B,f)$, it is a perfect $K(B,f)$-dg module. Therefore, in $\sing{B,f}$ we have that
\[
(E,d,h)\simeq cone(\phi) = cone(\varphi) \simeq target(\varphi).
\]
The cohomology groups in degree $-1$ and $0$ of $target(\varphi)$ vanish. Therefore, we have found that in $\sing{B,f}$ $(E,d,h)$ is equivalent to 
\begin{equation}\label{[-2n+1,-2]}
\begindc{\commdiag}[18]
\obj(50,-50)[4]{$\small{\underbracket{E_{-2n+3}\oplus E_{-2n+1}}_{-2n+1}}$}
\obj(130,-50)[5]{$\small{\underbracket{E_{-2n+4}\oplus E_{-2n+3}\oplus E_{-2n+2}}_{-2n+2}}$}
\obj(180,-50)[]{$\small{\dots}$}
\obj(230,-50)[]{$\underbracket{K}_{-2},$}
\mor(68,-49)(102,-49){$\small{\begin{bmatrix}
d_{-2n+3} & 0 \\
-f & 0 \\
h_{-2n+3} & d_{-2n+1}
\end{bmatrix}}$}[\atright, \solidarrow]
\mor(102,-45)(68,-45){$\small{\begin{bmatrix}
0 & -1 & 0 \\
0 & 0 & h_{-2n+2}
\end{bmatrix}}$}[\atright, \solidarrow]
\mor(158,-49)(178,-49){$$}[\atright, \solidarrow]
\mor(178,-45)(158,-45){$$}[\atright, \solidarrow]
\mor(182,-49)(220,-49){$\small{\begin{bmatrix}
d_{-1} & 0 & 0 \\
-f & -d_{-2} & 0 \\
h_{-1} & 1 & d_{-3}
\end{bmatrix}}$}[\atright, \solidarrow]
\mor(220,-45)(182,-45){$\small{\begin{bmatrix}
0 & -1 & 0 \\
0 & 0 & 0 \\
0 & 0 & h_{-2}
\end{bmatrix}}$}[\atright, \solidarrow]
\enddc
\end{equation}
where
\[
K=Ker\Bigl (\begin{bmatrix}
-f & -d_{-1} & 0 \\
h_0 & 1 & d_{-2}
\end{bmatrix}\Bigr ).
\]

This is still an element in $\textup{Coh}^s(B,f)$. Indeed, from the short exact sequence of $B$-modules
\[
0 \rightarrow Ker(\begin{bmatrix} 1 & d_{-1} \end{bmatrix}) \rightarrow E_0\oplus E_{-1} \xrightarrow{\begin{bmatrix} 1 & d_{-1} \end{bmatrix}} E_0 \rightarrow 0,
\]
since $E_0$ and $E_{-1}$ are $B$-projective, we conclude that $Ker(\begin{bmatrix} 1 & d_{-1} \end{bmatrix})$ is $B$-projective too. As the complex $cone(\varphi)$ is exact in degree $-1$, we also have the following short exact sequence of $B$-modules:
\[
0 \rightarrow K \rightarrow E_{0} \oplus E_{-1} \oplus E_{-2} \xrightarrow{\begin{bmatrix} -f & -d_{-1} & 0 \\ h_0 & 1 & d_{-2} \end{bmatrix}}  \underbracket{Im\Bigl ( \begin{bmatrix} -f & -d_{-1} & 0 \\ h_0 & 1 & d_{-2} \end{bmatrix} \Bigr )}_{= Ker \Bigl ( \begin{bmatrix} 1 & d_{-1} \end{bmatrix} \Bigr )} \rightarrow 0.
\]
As $E_0$, $E_{-1}$, $E_{-2}$ and $Ker \Bigl ( \begin{bmatrix} 1 & d_{-1} \end{bmatrix} \Bigr )$ are projective $B$-modules, we conclude. 

Notice  that the $K(B,f)$-dg module (\ref{[-2n+1,-2]}) can be written as the cone of the following morphism of $K(B,f)$-dg modules:
\begin{equation}
\begindc{\commdiag}[18]
\obj(50,50)[1]{$\small{\underbracket{E_{-2n+3}}_{-2n+2}}$}
\obj(130,50)[2]{$\small{\underbracket{E_{-2n+3}}_{-2n+3}}$}
\mor(55,51)(125,51){$\small{f}$}[\atright, \solidarrow]
\mor(125,55)(55,55){$\small{1}$}[\atright, \solidarrow]
\obj(0,0)[3]{$\small{\underbracket{E_{-2n+1}}_{-2n+1}}$}
\obj(50,0)[4]{$\small{\underbracket{E_{-2n+4}\oplus E_{-2n+2}}_{-2n+2}}$}
\obj(130,0)[5]{$\small{\underbracket{E_{-2n+5}\oplus E_{-2n+4}\oplus E_{-2n+3}}_{-2n+3}}$}
\obj(177,0)[]{$\small{\dots}$}
\obj(200,0)[]{$\underbracket{K}_{-2}.$}
\mor(5,1)(35,1){$\small{\begin{bmatrix}
0 \\
d_{-2n+1}
\end{bmatrix}}$}[\atright, \solidarrow]
\mor(35,5)(5,5){$\small{\begin{bmatrix}
0 & h_{-2n+2}
\end{bmatrix}}$}[\atright, \solidarrow]
\mor(67,1)(100,1){$\small{\begin{bmatrix}
d_{-2n+4} & 0 \\
-f & 0 \\
h_{-2n+4} & d_{-2n+2}
\end{bmatrix}}$}[\atright, \solidarrow]
\mor(100,5)(67,5){$\small{\begin{bmatrix}
0 & -1 & 0 \\
0 & 0 & h_{-2n+3}
\end{bmatrix}}$}[\atright, \solidarrow]
\mor(160,1)(176,1){$$}[\atright, \solidarrow]
\mor(176,5)(160,5){$$}[\atright, \solidarrow]
\mor(179,1)(195,1){$$}[\atright, \solidarrow]
\mor(195,5)(179,5){$$}[\atright, \solidarrow]
\mor{1}{4}{$\small{\begin{bmatrix}
d_{-2n+3} \\
h_{-2n+3}
\end{bmatrix}}$}[\atright, \solidarrow]
\mor{2}{5}{$\small{\begin{bmatrix}
0 \\
-d_{-2n+3} \\
1
\end{bmatrix}}$}
\enddc
\end{equation}
As the source of this morphism is $\underbracket{E_{-2n+3}}_{-2n+3}\otimes_B K(B,f)$, and $E_{-2n+3}$ is a perfect $B$-module, this morphism provides an equivalence between $(E,d,h)$ and the target in $\sing{B,f}$. Moreover, we can iterate this procedure: the target of this morphism can be written as the cone of the following morphism:
\begin{equation}
\begindc{\commdiag}[18]
\obj(50,50)[1]{$\small{\underbracket{E_{-2n+4}}_{-2n+3}}$}
\obj(130,50)[2]{$\small{\underbracket{E_{-2n+4}}_{-2n+4}}$}
\mor(55,51)(125,51){$\small{f}$}[\atright, \solidarrow]
\mor(125,55)(55,55){$\small{1}$}[\atright, \solidarrow]
\obj(-30,0)[3]{$\small{\underbracket{E_{-2n+1}}_{-2n+1}}$}
\mor(-25,1)(-5,1){$\small{d_{-2n+1}}$}[\atright, \solidarrow]
\mor(-5,5)(-25,5){$\small{h_{-2n+2}}$}[\atright, \solidarrow]
\obj(0,0)[3]{$\small{\underbracket{E_{-2n+2}}_{-2n+2}}$}
\obj(50,0)[4]{$\small{\underbracket{E_{-2n+5}\oplus E_{-2n+3}}_{-2n+3}}$}
\obj(130,0)[5]{$\small{\underbracket{E_{-2n+6}\oplus E_{-2n+5}\oplus E_{-2n+4}}_{-2n+4}}$}
\obj(177,00)[]{$\small{\dots}$}
\obj(195,0)[]{$\underbracket{K}_{-2}.$}
\mor(5,1)(35,1){$\small{\begin{bmatrix}
0 \\
d_{-2n+2}
\end{bmatrix}}$}[\atright, \solidarrow]
\mor(35,5)(5,5){$\small{\begin{bmatrix}
0 & h_{-2n+3}
\end{bmatrix}}$}[\atright, \solidarrow]
\mor(67,1)(100,1){$\small{\begin{bmatrix}
d_{-2n+5} & 0 \\
-f & 0 \\
h_{-2n+5} & d_{-2n+3}
\end{bmatrix}}$}[\atright, \solidarrow]
\mor(100,5)(67,5){$\small{\begin{bmatrix}
0 & -1 & 0 \\
0 & 0 & h_{-2n+4}
\end{bmatrix}}$}[\atright, \solidarrow]
\mor(160,1)(176,1){$$}[\atright, \solidarrow]
\mor(176,5)(160,5){$$}[\atright, \solidarrow]
\mor(179,1)(195,1){$$}[\atright, \solidarrow]
\mor(195,5)(179,5){$$}[\atright, \solidarrow]
\mor{1}{4}{$\small{\begin{bmatrix}
d_{-2n+4} \\
h_{-2n+4}
\end{bmatrix}}$}[\atright, \solidarrow]
\mor{2}{5}{$\small{\begin{bmatrix}
0 \\
-d_{-2n+4} \\
1
\end{bmatrix}}$}
\enddc
\end{equation}
Once again, as the source of this morphism of $K(B,f)$-dg modules is perfect, we obtain an equivalence between $(E,d,h)$ and the target of the morphism in $\sing{B,f}$. Proceeding this way, we obtain a chain of equivalences between our initial $K(B,f)$-dg module and the following:
\begin{equation}\label{ci siamo quasi}
\begindc{\commdiag}[20]
\obj(30,0)[]{$\small{\underbracket{E_{-2n+1}}_{-2n+1}}$}
\obj(60,0)[]{$\small{\underbracket{E_{-2n+2}}_{-2n+2}}$}
\obj(90,0)[]{$\small{\dots}$}
\obj(120,0)[]{$\small{\underbracket{E_{-4}}_{-4}}$}
\obj(150,0)[]{$\small{\underbracket{E_{-1}\oplus E_{-3}}_{-3}}$}
\obj(195,0)[]{$\underbracket{K}_{-2}.$}
\mor(35,1)(55,1){$\small{d_{-2n+1}}$}[\atright, \solidarrow]
\mor(55,5)(35,5){$\small{h_{-2n+2}}$}[\atright, \solidarrow]
\mor(65,1)(85,1){$\small{d_{-2n+2}}$}[\atright, \solidarrow]
\mor(85,5)(65,5){$\small{
h_{-2n+3}}$}[\atright, \solidarrow]
\mor(95,1)(115,1){$\small{d_{-5}}$}[\atright, \solidarrow]
\mor(115,5)(95,5){$\small{h_{-4}}$}[\atright, \solidarrow]
\mor(122,1)(142,1){$\small{\begin{bmatrix}
0 \\
d_{-4}
\end{bmatrix}}$}[\atright, \solidarrow]
\mor(142,5)(122,5){$\small{\begin{bmatrix}
0 & h_{-3}
\end{bmatrix}}$}[\atright, \solidarrow]
\mor(158,1)(185,1){$\small{\begin{bmatrix}
d_{-1} & 0 \\
-f & 0 \\
h_{-1} & d_{-3}
\end{bmatrix}}$}[\atright, \solidarrow]
\mor(185,5)(158,5){$\small{\begin{bmatrix}
0 & -1 & 0 \\
0 & 0 & h_{-2}
\end{bmatrix}}$}[\atright, \solidarrow]
\enddc
\end{equation}
Notice that the $K(B,f)$-dg module (\ref{ci siamo quasi}) is concentrated in $2n-1$ degrees. We can therefore apply the induction hypothesis to conclude that this $K(B,f)$-dg module is  naturally equivalent, in $\sing{B,f}$, to 
\begin{equation}
    \begindc{\commdiag}[20]
    \obj(0,0)[1]{$\small{\bigoplus_{i\in \mathbb{Z}}E_{2i-1}}$}
    \obj(145,0)[2]{$\small{E_{-2n+2}\oplus E_{-2n+4} \oplus \dots \oplus E_{-4}\oplus K}.$}
    \mor(10,-1)(110,-1){$\small{\begin{bmatrix}
d_{-2n+1} & h_{-2n+3} & \dots & 0 & 0 \\
0 & d_{-2n+3} & \dots & 0 & 0 \\
 & & \vdots & & \\
 0 & 0 & \dots & h_{-3} & 0 \\
 0 & 0 & \dots & 0 & d_{-1} \\
 0 & 0 & \dots & 0 & -f \\
 0 & 0 & \dots & d_{-3} & h_{-1}
\end{bmatrix}}$}[\atright, \solidarrow]
\mor(110,1)(10,1){$\small{\begin{bmatrix}
h_{-2n+2} & 0 & \dots & 0 & 0 & 0 \\
d_{-2n+2} & h_{-2n+4} & \dots & 0 & 0 & 0 \\
 & & \vdots & &  &\\
 0 & 0 & \dots & 0 & 0 & h_{-2} \\
 0 & 0 & \dots & 0 & -1 & 0
\end{bmatrix}}$}[\atright, \solidarrow]
    \enddc
\end{equation}
Recall that $K=Ker\Bigl (\begin{bmatrix}
-f & -d_{-1} & 0 \\
h_0 & 1 & d_{-2}
\end{bmatrix}\Bigr )\subseteq E_0\oplus E_1 \oplus E_2$. We can finally consider the following morphism of $K(B,f)$-dg modules concentrated in degrees $-1$ and $0$
\begin{equation}\label{final step}
\begindc{\commdiag}[20]
\obj(0,60)[1]{$\small{\bigoplus_{i\in \mathbb{Z}}E_{2i-1}}$}
\obj(0,0)[2]{$\small{\bigoplus_{i\in \mathbb{Z}}E_{2i}}$}
\obj(120,60)[3]{$\small{\bigoplus_{i\in \mathbb{Z}}E_{2i-1}}$}
\obj(120,0)[4]{$\small{E_{-2n+2}\oplus E_{-2n+4} \oplus \dots \oplus E_{-4}\oplus K}.$}
\mor(-2,55)(-2,5){$\small{d+h}$}[\atright, \solidarrow]
\mor(2,5)(2,55){$\small{d+h}$}[\atright, \solidarrow]
\mor(118,55)(118,5){$\small{\begin{bmatrix}
d_{-2n+1} & h_{-2n+3} & \dots & 0 & 0 \\
0 & d_{-2n+3} & \dots & 0 & 0 \\
 & & \vdots & & \\
 0 & 0 & \dots & h_{-3} & 0 \\
 0 & 0 & \dots & 0 & d_{-1} \\
 0 & 0 & \dots & 0 & -f \\
 0 & 0 & \dots & d_{-3} & h_{-1}
\end{bmatrix}}$}[\atright, \solidarrow]
\mor(122,5)(122,55){$\small{\begin{bmatrix}
h_{-2n+2} & 0 & \dots & 0 & 0 & 0 \\
d_{-2n+2} & h_{-2n+4} & \dots & 0 & 0 & 0 \\
 & & \vdots & &  &\\
 0 & 0 & \dots & 0 & 0 & h_{-2} \\
 0 & 0 & \dots & 0 & -1 & 0
\end{bmatrix}}$}[\atright, \solidarrow]
\mor{2}{4}{$\small{\begin{bmatrix}
1 & 0 & \dots & 0 & 0 & 0 \\
0 & 1 & \dots & 0 & 0 & 0 \\
 & & \vdots &  &  & \\
 0 & 0 & \dots & 1 & 0 & 0 \\
 0 & 0 & \dots & 0 & 0 & 1 \\
 0 & 0 & \dots & 0 & -d_{-2} & -h_0 \\
  0 & 0 & \dots & 0 & 1 & 0 \\
\end{bmatrix}}$}[\atright,\solidarrow]
\mor{1}{3}{$id$}[\atright, \solidarrow]
\enddc
\end{equation}
It is not hard to check that morphism (\ref{final step}) is a quasi-isomorphism. Notice that the target of (\ref{final step}) is equivalent in $\sing{B,f}$ to the $K(B,f)$-dg module $(E,d,h)$.

Also notice that since all the passages above are functorial, the equivalence is natural in $(E,d,h)$. In particular, a morphism of $K(B,f)$-dg modules $\phi : (E,d,h)\rightarrow (E',d',h')$ corresponds, under this equivalence, to
\[
\begindc{\commdiag}
\obj(0,30)[1]{$\small{\bigoplus_{i\in \mathbb{Z}}E_{2i-1}}$}
\obj(40,30)[2]{$\small{\bigoplus_{i\in \mathbb{Z}}E_{2i}}$}
\obj(0,0)[3]{$\small{\bigoplus_{i\in \mathbb{Z}}E'_{2i-1}}$}
\obj(40,0)[4]{$\small{\bigoplus_{i\in \mathbb{Z}}E'_{2i}}.$}
\mor(10,30)(30,30){$d+h$}[\atright,\solidarrow]
\mor(30,35)(10,35){$d+h$}[\atright,\solidarrow]
\mor(10,0)(30,0){$d'+h'$}[\atright,\solidarrow]
\mor(30,5)(10,5){$d'+h'$}[\atright,\solidarrow]
\mor{1}{3}{$\oplus \phi_{2i-1}$}
\mor{2}{4}{$\oplus \phi_{2i}$}
\enddc
\]
\end{proof}
\begin{rmk}
The algorithm we have provided actually puts the final $K(B,f)$-dg module
\[
\begindc{\commdiag}[20]
\obj(0,100)[1]{$\small{E_{-2n+1}\oplus E_{-2n+3}\oplus \dots \oplus E_{-3}\oplus E_{-1}}$}
\obj(120,100)[2]{$\small{E_{-2n+2}\oplus E_{-2n+4} \oplus \dots \oplus E_{-2}\oplus E_{0}}$}
\mor(50,100)(70,100){$d+h$}[\atright, \solidarrow]
\mor(70,105)(50,105){$d+h$}[\atright, \solidarrow]
\enddc
\]
in degrees $-2n+1$ and $-2n+2$. However, thanks to Lemma \ref{lemma salto}, this is equivalent in $\sing{B,f}$ to the same dg module concentrated in degrees $-1$ and $0$.
\end{rmk}
\begin{cor}\label{cones Sing(B,f)}
Let $\phi_. :\bigl ( \dbarrow{E_{-1}}{E_0}{d}{h}\bigr )\rightarrow \bigl (\dbarrow{E'_{-1}}{E'_0}{d'}{h'}\bigr )$ be a closed morphism in $\textup{Coh}^s(B,f)$. Then $cone(\phi_.)$ is equivalent to $\dbarrow{\small{E'_{-1}\oplus E_0}}{\small{E'_0\oplus E_{-1}}}{\begin{bmatrix} d' & \phi_0 \\ 0 & -h\end{bmatrix}}{\begin{bmatrix} h' & \phi_{-1}\\ 0 & -d\end{bmatrix}}$ in $\sing{B,f}$.
\end{cor}
\begin{proof}
This is a straightforward consequence of the computation of $Cone(\phi)$ in $\textup{Coh}^s(B,f)$ and of the previous theorem. 
\end{proof}
\begin{cor}\label{Orlov equivalence}
The lax monoidal $\oo$-natural transformation
\begin{equation}
Orl^{-1,\otimes} : \sing{\bullet,\bullet} \rightarrow \mf{\bullet}{\bullet} : \lgm{1}^{op,\boxplus}\rightarrow \dgcatmo
\end{equation}
constructed in \cite[\S2.4]{brtv} defines a lax monoidal $\oo$-natural equivalence.
\end{cor}
\begin{proof}
By Kan extension and descent, it is sufficient to consider the affine case. Let $(Spec(B),f)\in \lgm{1}^{\textup{aff},op}$.
As the dg categories $\sing{B,f}$ and $\mf{B,f}$ are triangulated, it is sufficient to show that the induced functor
\[
[Orl^{-1}] : [\sing{B,f}] \rightarrow [\mf{B}{f}]
\]
\[
(E,d,h)\mapsto \bigl (\dbarrow{\bigoplus_{i\in \mathbb{Z}}E_{2i-1}}{\bigoplus_{i\in \mathbb{Z}}E_{2i}}{d+h}{d+h}\bigr )
\]
is an equivalence. Consider
\[
Orl : [\mf{B}{f}] \rightarrow [\sing{B,f}]
\] 
\[
\bigl (\dbarrow{E}{F}{p}{q} \bigr ) \mapsto \bigl (\dbarrow{\underbracket{E}_{-1}}{\underbracket{F}_{0}}{p}{q}\bigr ).
\]
This is an exact functor between triangulated categories by Corollary \ref{suspension in Sing} and by Corollary \ref{cones Sing(B,f)}.
It is clear that $[Orl^{-1}]\circ Orl$ is the identity functor. By Theorem \ref{structure theorem n=1}, $Orl\circ [Orl^{-1}]$ is equivalent to the identity functor too.
\end{proof}
\begin{rmk}
Notice that $Orl$ is a derived version of the "Cok" functor introduced in \cite{orl04}. Indeed, when $f$ is flat, the $K(B,f)$-dg module $coker(p)$ concentrated in degree $0$ is quasi-isomorphic to $\dbarrow{\underbracket{E}_{-1}}{\underbracket{F}_{0}}{p}{q}$.
\end{rmk}
\begin{rmk}
In \cite{efpo15}, the authors also introduced a coherent version of $\mf{B}{f}$. When $f$ is flat, they proved it to be equivalent to another category of singularities, defined as the Verdier quotient
\begin{equation}
\sing{B,f}_{coh}=\cohb{B/f}/\textup{\textbf{E}}
\end{equation}
where $\textup{\textbf{E}}$ is the thick subcategory of $\cohb{B/f}$ generated by the image of the pullback $\iota^*: \cohb{B}\rightarrow \cohb{B/f}$. 

Our proof of Theorem \ref{structure theorem n=1} also tells us that, for any $f$, all objects in this triangulated category can be represented by $K(B,f)$-dg modules concentrated in degrees $[-1,0]$. This can be used to show that the equivalence proven in \cite{efpo15} holds for any potential $f$, provided that we consider the derived fiber instead of $B/f$.
\end{rmk}

Massimo.Pippi@mathematik.uni-regensburg.de

Fakult\"at f\"ur Mathematik, Universit\"at Regensburg, 93040 Regensburg, Germany
\end{document}